\RequirePackage{fix-cm}
\RequirePackage{amsthm}
\documentclass[sn-mathphys-num,pdflatex]{sn-jnl}

\usepackage{graphicx}%
\usepackage{multirow}%
\usepackage{amsmath,amssymb,amsfonts}%
\usepackage{mathrsfs}%
\usepackage[title]{appendix}%
\usepackage{xcolor}%
\usepackage{textcomp}%
\usepackage{manyfoot}%
\usepackage{booktabs}%
\usepackage{algorithm}%
\usepackage{algorithmicx}%
\usepackage{algpseudocode}%
\usepackage{listings}%
\theoremstyle{thmstyleone}%
\newtheorem{theorem}{Theorem}% 

\theoremstyle{thmstyletwo}%
\newtheorem{example}{Example}%
\newtheorem{remark}{Remark}%

\theoremstyle{thmstylethree}%
\newtheorem{definition}{Definition}%

\usepackage{hyperref}
\usepackage{lmodern}

\newtheorem{lemma}{Lemma}
\newtheorem{corollary}{Corollary}

\newcommand{\R}{\mathbb{R}}
\newcommand{\N}{\mathbb{N}}

\newcommand{\Bcl}{\bar{B}}
\newcommand{\keps}{\kappa_\eps}

\newcommand{\deriv}{\mathrm{d}}
\newcommand{\ldDH}{\mathrm{\underline{d}}}
\newcommand{\eps}{\varepsilon}

\DeclareMathOperator*{\conv}{conv}

\DeclareMathOperator*{\argmin}{arg\,min}

\begin{document}

\title{Analyzing the speed of convergence in nonsmooth optimization via the Goldstein subdifferential with application to descent methods}

\author*[1]{\fnm{Bennet} \sur{Gebken}}\email{bennet.gebken@cit.tum.de}

\affil*[1]{\orgdiv{Department of Mathematics}, \orgname{Technical University of Munich}, \orgaddress{\street{Boltzmannstr. 3}, \city{Garching b. München}, \postcode{85748}, \country{Germany}}}

\abstract{The Goldstein $\eps$-subdifferential is a relaxed version of the Clarke subdifferential which has recently appeared in several algorithms for nonsmooth optimization. With it comes the notion of $(\eps,\delta)$-critical points, which are points in which the element with the smallest norm in the $\eps$-subdifferential has norm at most $\delta$. To obtain points that are critical in the classical sense, $\eps$ and $\delta$ must vanish. In this article, we analyze at which speed the distance of $(\eps,\delta)$-critical points to the minimum vanishes with respect to $\eps$ and $\delta$. Afterwards, we apply our results to gradient sampling methods and perform numerical experiments. Throughout the article, we put a special emphasis on supporting the theoretical results with simple examples that visualize them.}


\maketitle

\section{Introduction}

Theoretical analysis of the speed of convergence of algorithms is an important part of optimization, since numerical examples and benchmarks alone may not capture the entire behavior. This is especially true in nonsmooth optimization, where functions cannot be locally linearized and may instead have complicated and diverse kink structures. In smooth optimization, solution methods with (at least) superlinear convergence, like Newtons method or Quasi-Newton methods, belong to the state of the art. In nonsmooth optimization on the other hand, to the best of the authors' knowledge, methods with provably superlinear convergence have yet to be found, and they were even called a ``wondrous grail'' in \cite{MS2012}. Here, by nonsmooth optimization we mean the minimization of a locally Lipschitz continuous function $f : \R^n \rightarrow \R$ without assuming additional structure like an analytical expression or certain geometrical properties for the set of nonsmooth points. There are severe challenges when analyzing convergence for such general functions: Firstly, $f$ cannot be approximated locally in terms of a polynomial since Taylor expansion is not available. Secondly, the necessary optimality condition $\nabla f(x^*) = 0$ from smooth optimization, which is a nonlinear system of equations, generalizes to the discontinuous inclusion $0 \in \partial f(x^*)$, where $\partial f$ is the Clarke subdifferential \cite{C1990}.
Due to these challenges, the convergence analysis from smooth optimization cannot be generalized to nonsmooth optimization in a straight forward way. 

In practice, the Clarke subdifferential brings an additional challenge: Since $\partial f(x) = \{ \nabla f(x) \}$ whenever $f$ is continuously differentiable in $x$, and since the set of points in which $f$ is not continuously differentiable is typically a null set, the Clarke subdifferential is not a practical way to capture the nonsmoothness of $f$. One way to solve this issue is the usage of the Goldstein $\eps$-subdifferential $\partial_\eps f(x)$ \cite{G1977}, which is the convex hull of the union of all Clarke subdifferentials from a closed $\eps$-ball around $x$. Unless $0 \in \partial_\eps f(x)$, the element $v$ in $-\partial_\eps f(x)$ with the smallest norm is a descent direction for $f$ at $x$.  Computing an approximation of $v$ via an approximation of $\partial_\eps f(x)$ is the basic idea of so-called gradient sampling methods \cite{BLO2005,BCL2020,GP2021,MY2012}. In these methods, once $\| v \|$ lies below a threshold $\delta$, the method reduces the values of $\eps$ and $\delta$ and then continues. In this way, for vanishing sequences $(\eps_j)_j$ and $(\delta_j)_j$, a sequence $(x^j)_j$ is generated that satisfies $\min(\| \partial_{\eps_j} f(x^j) \|) \leq \delta_j$ for all $j \in \N$. By upper semicontinuity of $\partial f$, this implies $0 \in \partial f(x^*)$.

The above descent strategy motivates our main question in this article:
\begin{center}
    \fbox{
        \begin{minipage}{300pt}
            Let $(x^j)_j \in \R^n$, $(\eps_j)_j \in \R^{\geq 0}$ and $(\delta_j)_j \in \R^{\geq 0}$ such that $x^j \rightarrow x^* \in \R^n$, $\eps_j \rightarrow 0$, $\delta_j \rightarrow 0$ and $\min(\| \partial_{\eps_j} f(x^j) \|) \leq \delta_j$ for all $j \in \N$. What is the relationship between the speed of convergence of $(x^j)_j$ and the speeds of $(\eps_j)_j$ and $(\delta_j)_j$?
        \end{minipage}
    }
\end{center}
While our motivation for this question comes from gradient sampling, we will analyze it from a purely abstract point of view, without gradient sampling or any other specific solver in mind. Our main result (and an answer to this question) is Theorem \ref{thm:xj_speed_higher_order}, which states that if $f$ satisfies a $p$-order growth condition (cf. Definition \ref{def:order_minimum}) and a $p$-order semismoothness property (cf. \eqref{eq:higher_order_semismooth}) in $x^*$, then $\| x^j - x^* \|$ is bounded by the maximum of $M \eps_j^{1/p}$ and $M \delta_j^{1/(p-1)}$ for some constant $M > 0$. For $p = 1$, in which case $f$ grows linearly around $x^*$ (and $x^*$ is sometimes referred to as a sharp minimum), the $p$-order semismoothness property is implied by standard semismoothness \cite{M1977}, which covers large classes of functions like convex functions and piecewise differentiable functions \cite{SS2008}. For $p > 1$, it is more challenging to verify. However, we show that it is satisfied for piecewise differentiable functions with a certain higher-order convexity property (cf. Lemma \ref{lem:convex_higher_order}). The obvious application of our results is the analysis of gradient sampling methods and, in particular, their linear convergence. But since our results could also be used to show superlinear or even faster convergence, they may also inspire entirely new methods for nonsmooth optimization.

We are not aware of previous work on our main question in this general form. Nonetheless, similarities can be found when looking at analyses for specific solvers, especially in recent years: In \cite{HSS2017}, linear convergence of the (nonnormalized) random gradient sampling algorithm \cite{BLO2005,K2007,BCL2020} is proven for twice piecewise differentiable max-type functions. The requirements for this include assumptions on the $\mathcal{V}$ and $\mathcal{U}$-spaces (cf. \cite{LOS1999}) of the objective, on the size of the sampling radius $\eps$ w.r.t. the (unknown) minimum $x^*$, affine independence of the gradients and positive definiteness of a weighted Hessian matrix of the selection functions. These assumptions imply quadratic growth around the minimum. However, an analogue of our higher-order semismoothness property does not seem to be required. In \cite{DJ2024}, a descent method based on the Goldstein $\eps$-subdifferential with random sampling is proposed with ``nearly linear'' convergence. The requirements include quadratic growth and a certain smooth substructure around the minimum, and their analysis is based on a certain gradient inequality for the $\eps$-subdifferential. This inequality is similar to the well-known Kurdyka-\L ojasiewicz inequality, which is an important tool for analyzing the speed of convergence for proximal methods \cite{ABS2011,KNS2016,BMM2024}. In \cite{CD2023}, a subgradient method was proposed and analyzed which achieves superlinear convergence for ``uniformly semismoothness'' functions with sharp minima. Finally, in \cite{ZLJ2020}, a descent method based on the Goldstein $\eps$-subdifferential with random sampling and fixed $\eps$ and $\delta$ was proposed, which computes points satisfying $\min(\| \partial_\eps f(x) \|) \leq \delta$ in a finite number of iterations. Due to the finiteness, the authors were able to carry out a non-asymptotic convergence rate analysis in terms of the number of gradient oracle calls in relation to $\eps$ and $\delta$. (Further non-asymptotic analyses can be found in \cite{S2020,JKL2023}.) A nice summary on other recent results in nonsmooth nonconvex optimization can be found in \cite{JKL2023}.

The remainder of this article is structured as follows: Starting off, in Section \ref{sec:basics}, we introduce our notation and the basics of nonsmooth analysis that we need throughout the article. Afterwards, in Section \ref{sec:motivation}, we present three examples that motivate the polynomial growth and the higher-order semismoothness assumption we require for our main result. Since the higher-order semismoothness assumption appears to be novel, we analyze which classes of functions possess this property in Section \ref{sec:higher_order_semismooth}. Subsequently, we prove our main result in Section \ref{sec:analysis_speed_of_convergence}. As an application, in Section \ref{sec:descent_methods}, we consider descent methods based on the Goldstein $\eps$-subdifferential. The Matlab code for our numerical experiments, including an implementation of the deterministic gradient sampling method described in \cite{GP2021,G2022,G2024}, is freely available at \url{https://github.com/b-gebken/DGS}. Finally, in Section \ref{sec:outlook}, we summarize our results and discuss possible directions for future research.

\section{Preliminaries} \label{sec:basics}

    In this section, we briefly introduce the basics of nonsmooth analysis. For a more thorough introduction, we refer to \cite{C1990}. To this end, let $f : \R^n \rightarrow \R$ be locally Lipschitz continuous and let $\Omega$ be the set of points in which $f$ is not differentiable. By Rademacher's Theorem, $\Omega$ is a null set. The \emph{Clarke subdifferential of} $f$ \emph{at} $x \in \R^n$ is defined as
    \begin{equation*}
        \begin{aligned}
            \partial f(x) := \conv &\left( \left\{ \xi \in \R^n : \exists (x^j)_j \in \R^n \setminus \Omega \text{ with } \lim_{j \rightarrow \infty} x^j = x \text{ and } \right. \right. \\
            &\left. \left. \lim_{j \rightarrow \infty} \nabla f(x^j) = \xi \right\} \right),
        \end{aligned}
    \end{equation*}
    and its elements are called \emph{subgradients}. If $x$ is a point with $0 \in \partial f(x)$, then it is called \emph{critical}, which is a necessary condition for optimality. The Clarke subdifferential satisfies the following mean value theorem: For $x,y \in \R^n$ there is some $s \in (0,1)$ and some $\xi \in \partial f(x + s (y - x))$ such that
    \begin{align} \label{eq:mean_value_theorem}
        f(y) - f(x) = \langle \xi, y - x \rangle.
    \end{align}
    To circumvent some of the practical issues of the Clarke subdifferential (cf. \cite{BGLS2006}, Section 9.1), the \emph{Goldstein} $\eps$\emph{-subdifferential} may be used instead. For $x \in \R^n$ and $\eps \geq 0$, it is defined as the compact set
    \begin{align*}
        \partial_\eps f(x) := \conv \left( \bigcup_{y \in \Bcl_\eps(x)} \partial f(y) \right),
    \end{align*}
    where $\Bcl_\eps(x) := \{ y \in \R^n : \| y - x \| \leq \eps \}$ and $\| \cdot \|$ is the Euclidean norm. For $\eps, \delta \geq 0$ we say that $x$ is $(\eps,\delta)$\emph{-critical}, if $\min(\| \partial_\eps f(x) \|) \leq \delta$, i.e., if there is some $\xi \in \partial_\eps f(x)$ with $\|\xi \| \leq \delta$. Clearly, $(\eps,\delta)$-criticality is a weaker optimality condition than criticality. However, upper semicontinuity of $\partial f$ implies that if there are sequences $(x^j)_j \in \R^n$, $(\eps_j)_j, (\delta_j)_j \in \R^{\geq 0}$ with $x^j \rightarrow x^*$, $\eps_j \rightarrow 0$ and $\delta_j \rightarrow 0$ such that $x^j$ is $(\eps_j,\delta_j)$-critical for all $j \in \N$, then $x^*$ is critical (cf. Lemma 4.4.4 in \cite{G2022}). In addition to subdifferentials, we also require directional derivatives of $f$. To assure that they exist and are well-behaved, we occasionally assume that $f$ is \emph{semismooth} \cite{M1977}, which means that $f$ is locally Lipschitz continuous and for any $x \in \R^n$ and $d \in \R^n$, the limit
    \begin{align} \label{eq:def_semismooth}
        \lim_{\xi \in \partial f(x + t d'), t \searrow 0, d' \rightarrow d} \langle \xi, d' \rangle
    \end{align}
    exists. If $f$ is semismooth, then for any $d \in \R^n$, the \emph{directional derivative}
    \begin{align*}
        f'(x,d) = \lim_{t \searrow 0} \frac{f(x + t d) - f(x)}{t}
    \end{align*}
    exists and equals the limit \eqref{eq:def_semismooth}. A large class of semismooth functions are the \emph{piecewise $p$-times differentiable} functions \cite{S2012}, which are continuous functions $f : \R^n \rightarrow \R$ for which there are $p$-times continuously differentiable functions $f_1, \dots, f_m : \R^n \rightarrow \R$, called \emph{selection functions}, such that $f(x) \in \{ f_1(x), \dots, f_m(x) \}$ for all $x \in \R^n$. For such a function, the set $A(x) := \{ i \in \{1,\dots,m\} : f(x) = f_i(x) \}$ is called the \emph{active set}, and the Clarke subdifferential satisfies $\partial f(x) \subseteq \conv(\{ \nabla f_i(x) : i \in A(x) \})$.

    Finally, in situations where it is important to classify the speed of convergence, we use the standard concepts of \emph{Q-} and \emph{R-convergence}, which can be found in \cite{NW2006,UU2012}.

\section{Motivating the strategy} \label{sec:motivation}

    In this section, we consider simple examples that motivate the assumptions and concepts that we use throughout the article. To this end, assume that we are in the setting of our main question, i.e., let $(x^j)_j \in \R^n$, $(\eps_j)_j \in \R^{\geq 0}$ and $(\delta_j)_j \in \R^{\geq 0}$ such that $x^j \rightarrow x^* \in \R^n$, $\eps_j \rightarrow 0$, $\delta_j \rightarrow 0$ and $\min(\| \partial_{\eps_j} f(x^j) \|) \leq \delta_j$ for all $j \in \N$. We are interested in the question whether the speed of convergence of $(x^j)_j$ can be derived from the speeds of $(\eps_j)_j$ and $(\delta_j)_j$. Clearly, for arbitrary locally Lipschitz continuous functions $f$, this is not possible: In the trivial case where $f$ is constant, $(x^j)_j$ may converge arbitrarily slowly, independently from $(\eps_j)_j$ and $(\delta_j)_j$. Thus, we first have to analyze for which subclass of the class of locally Lipschitz continuous functions we even have a chance to obtain a positive answer. To derive the properties that these functions must have, we construct a series of examples where $(\eps_j)_j$ and $(\delta_j)_j$ may converge arbitrarily fast, but $(x^j)_j$ only converges slowly. 
    
    As a start, the case where $f$ is constant can be avoided by assuming that $x^*$ is a strict local minimum of $f$. Unfortunately, as the following example shows, this assumption is not enough:
    \begin{example} \label{example:exponential_decrease}
        Consider the function
        \begin{align*}
            f : \R^2 \rightarrow \R, \quad x \mapsto \max( \sigma(|x_1|), |x_2|),
        \end{align*}
        where $\sigma : \R^{\geq 0} \rightarrow \R$ is locally Lipschitz continuous with $\sigma(0) = 0$, $\sigma(t) > 0$ for all $t > 0$ and $\sigma|_{\R^{>0}}$ continuously differentiable. Then $f$ is locally Lipschitz continuous with the unique global minimum $x^* = 0$. For the set of nonsmooth points $\Omega$ of $f$ we have
        \begin{align*}
            \Omega \subseteq \{ (t,\sigma(|t|))^\top : t \in \R \} \cup \{ (t,-\sigma(|t|))^\top : t \in \R \}.
        \end{align*}
        Figure \ref{fig:example_exponential_decrease}(a) shows the graph of $f$ and the set of nonsmooth points for $\sigma$ as in \eqref{eq:example_exponential_decrease_sigma} below.\\
        \begin{figure}
            \centering
            \parbox[b]{0.49\textwidth}{
                \centering 
                \includegraphics[width=0.4\textwidth]{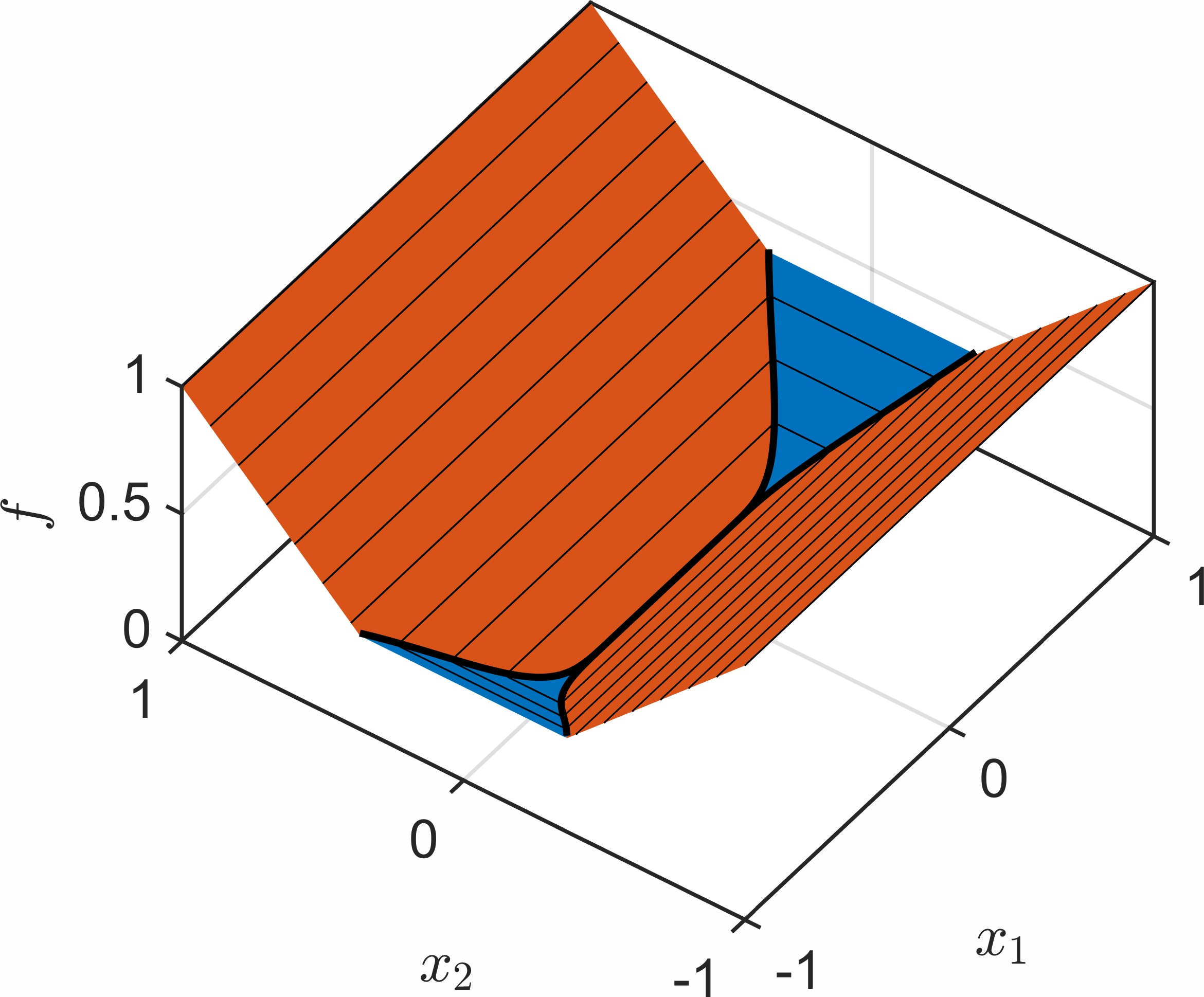}\\
                (a)
    		}
            \parbox[b]{0.49\textwidth}{
                \centering 
                \includegraphics[width=0.4\textwidth]{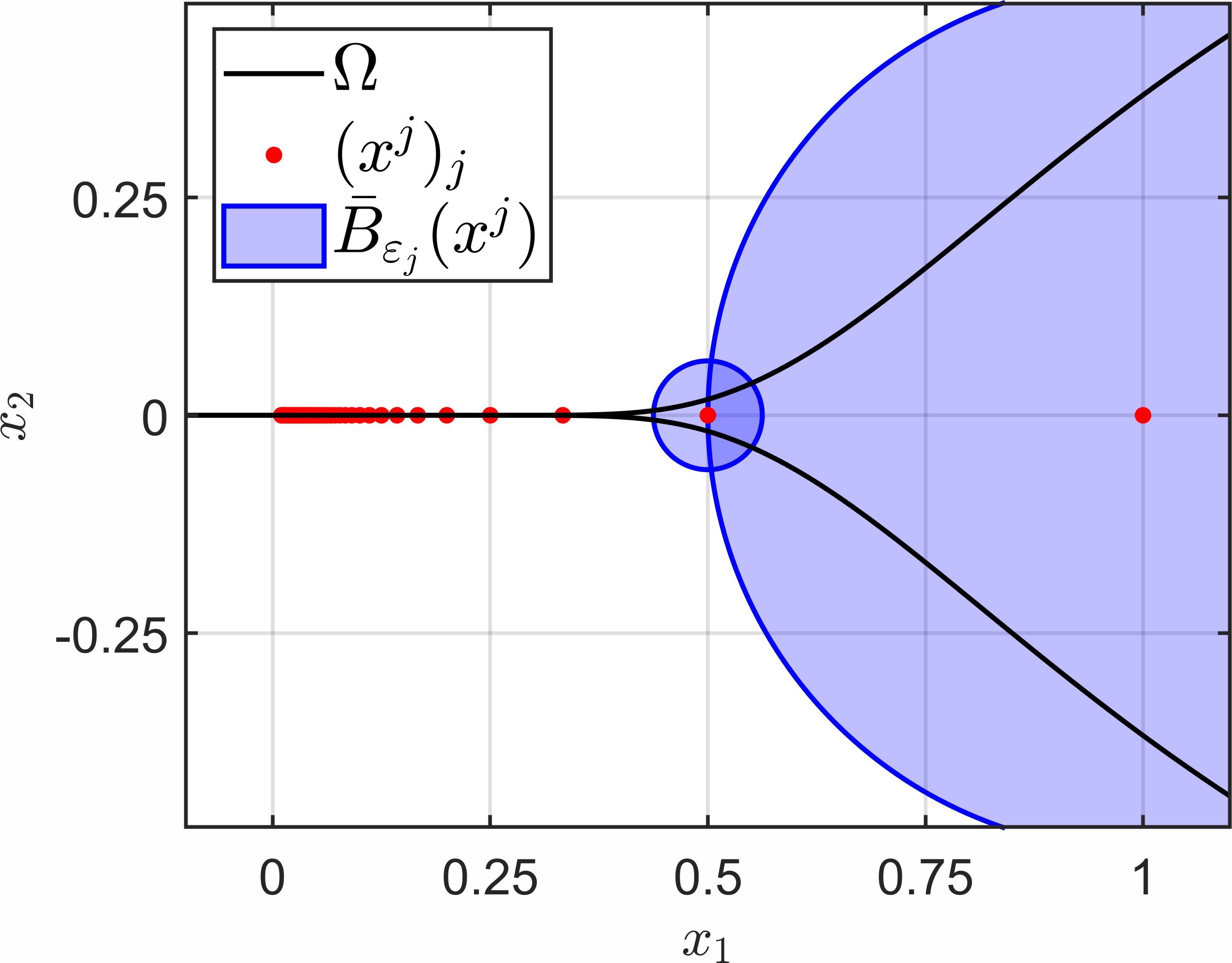}\\
                (b)
    		}
            \caption{(a) The graph of $f$ in Example \ref{example:exponential_decrease} with $\sigma$ as in \eqref{eq:example_exponential_decrease_sigma}. (b) The set of nonsmooth points $\Omega$, the sequence $(x^j)_j$ and the balls $\Bcl_{\eps_j}(x^j)$ in Example \ref{example:exponential_decrease}(ii).}
            \label{fig:example_exponential_decrease}
        \end{figure}%
        (i) Consider sequences $(x^j)_j$, $(\eps_j)_j$ and $(\delta_j)_j$ with
        \begin{align*}
            x^j = (j^{-1}, 0)^\top, \quad \eps_j = 0, \quad \delta_j > 0 \quad \forall j \in \N.
        \end{align*}
        Let $\rho : \R^{\geq 0} \rightarrow \R^{\geq 0}$ be continuous with $\rho(0) = 0$, $\rho(t) > 0$ for all $t > 0$ and $\rho(j) \leq \delta_j$ for all $j \in \N$. By the fundamental theorem of calculus,
        \begin{align*}
            \sigma(t) = \int_0^{t} \rho(1/s) ds
        \end{align*}
        satisfies the above assumptions for $\sigma$ with $\sigma'(t) = \rho(1/t)$ for all $t > 0$. By construction, we have
        \begin{align*}
            \min(\| \partial_{\eps_j} f(x^j) \|) 
            &= \| \{ \nabla f(x^j) \} \| 
            = \| (\sigma'(|x_1^j|),0)^\top \| \\
            &= \| (\rho(j),0)^\top \| 
            = \rho(j)
            \leq \delta_j 
            \quad \forall j \in \N.
        \end{align*}
        As an example, choosing $\rho(j) = \delta_j = 2 j^3 e^{-j^2}$ leads to 
        \begin{align} \label{eq:example_exponential_decrease_sigma}
            \sigma(t) = 
            \begin{cases}
                e^{-1/t^2}, & t > 0, \\
                0, & t = 0.
            \end{cases}
        \end{align}
        (ii) Consider sequences $(x^j)_j$, $(\eps_j)_j$ and $(\delta_j)_j$ with
        \begin{align*}
            x^j = (j^{-1}, 0)^\top, \quad \eps_j > 0, \quad \delta_j = 0 \quad \forall j \in \N.
        \end{align*}
        Let $\sigma$ be a function satisfying the above assumptions and additionally $\sigma(1/j) \leq \eps_j$ for all $j \in \N$. By construction, for all $x \in \R^n$ and $\eps > 0$ with $x_2 = 0$ and $\sigma(|x_1|) \leq \eps$, it holds $(0,1)^\top \in \partial_\eps f(x)$ and $(0,-1)^\top \in \partial_\eps f(x)$, so $\min(\| \partial_\eps f(x) \|) = 0$. In particular, $\min(\| \partial_{\eps_j} f(x^j) \|) = 0 \leq \delta_j$ for all $j \in \N$. As an example, for $\eps_j = 2^{-j^2}$ one may choose $\sigma$ again as in \eqref{eq:example_exponential_decrease_sigma}. A visualization of this case is shown in Figure \ref{fig:example_exponential_decrease}(b).
    \end{example}

    The previous example shows that for any sequences $(\eps_j)_j$ and $(\delta_j)_j$ (with at least one of them being nonzero), we can find a locally Lipschitz continuous function $f$ with a unique global minimum such that a sublinearly converging sequence $(x^j)_j$ satisfies $\min(\| \partial_{\eps_j} f(x^j) \|) \leq \delta_j$ for all $j \in \N$. Case (i) shows behavior that also occurs in the smooth case, since we chose $\eps_j = 0$ and the nonsmoothness of $f$ was irrelevant. Here, the issue is that $\min(\| \partial_{\eps_j} f(x^j) \|) = \| \nabla f(x^j) \|$ vanishes quickly even though $(x^j)_j$ only converges slowly, which is made possible by exponential decay of $f$ around $0$. Case (ii) highlights the additional behavior that exists in the nonsmooth case: Since the $\eps$-subdifferential is defined as the \emph{convex hull} of the subgradients, $\min(\| \partial_\eps f(x) \|)$ may be small (or, as in this example, even zero) although no subgradient at any point in $\Bcl_\eps(x)$ has a small norm. While this may also occur for smooth $f$, it is more prevalent for nonsmooth functions due to the jumps in the subgradients when crossing nonsmooth points in $\Omega$. In a way, this means that the geometry of $\Omega$ is relevant for the convergence behavior of $(x^j)_j$. More precisely, in Figure \ref{fig:example_exponential_decrease}(b), $\Omega$ forms a cusp that shrinks towards $(0,0)^\top$ with exponential speed. Thus, even for sublinearly converging $(x^j)_j$ and Q-superlinearly decreasing $(\eps_j)_j$, $\Bcl_{\eps_j}(x^j)$ may intersect $\Omega$ infinitely many times, which (in this example) implies $\min(\| \partial_{\eps_j} f(x^j) \|) = 0$. Furthermore, since either $(\eps_j)_j$ or $(\delta_j)_j$ were constantly zero in (i) and (ii), this example shows that even arbitrarily fast convergence of either $(\eps_j)_j$ or $(\delta_j)_j$ is not sufficient for fast convergence of $(x^j)_j$. Therefore, the behavior of both sequences has to be considered simultaneously to infer the convergence of $(x^j)_j$. 

    For smooth functions, the above behavior is well-known and can be avoided by assuming a certain rate of growth of $f$ when moving away from the minimum $x^*$, e.g., by assuming positive definiteness of the Hessian matrix at $x^*$. For nonsmooth functions, the following generalized concept can be used (cf. \cite{S1986,KGD2022}):
    \begin{definition} \label{def:order_minimum}
        A point $x^* \in \R^n$ is called a \emph{minimum of order} $p \in \N$ \emph{with constant} $\beta > 0$, if there is some open set $U \subseteq\R^n$ with $x^* \in U$ such that
        \begin{align*}
            f(x) \geq f(x^*) + \beta \| x - x^* \|^p \quad \forall x \in U.
        \end{align*}
    \end{definition}
    Note that a minimum of order $p$ is also a minimum of all orders greater than $p$. In Example \ref{example:exponential_decrease} (with $\sigma$ as in \eqref{eq:example_exponential_decrease_sigma}), it is easy to see that $x^*$ is not a minimum of any order, since $e^{-1/x_1^2}$ decreases faster than $\| x  - x^* \|^p = \| x \|^p$ for any $p \in \N$. Thus, restricting ourselves to functions with minimal points of a certain order allows us to avoid the issues highlighted in this example. However, this restriction is still not enough to be able to infer the speed of convergence of $(x^j)_j$, as the following example shows:

    \begin{example} \label{example:why_semismooth}
        Consider the function
        \begin{align*}
            f : \R \rightarrow \R, \quad x \mapsto 
            \begin{cases}
                x^2 \sin \left( \frac{1}{x} \right) + |x|, & x \neq 0,\\
                0, & x = 0,
            \end{cases}
        \end{align*}
        which is a modified version of a classical non-semismooth function from \cite{M1977}. The graph of $f$ is shown in Figure \ref{fig:example_why_semismooth}(a).
        \begin{figure}
            \centering
            \parbox[b]{0.49\textwidth}{
                \centering 
                \includegraphics[width=0.49\textwidth]{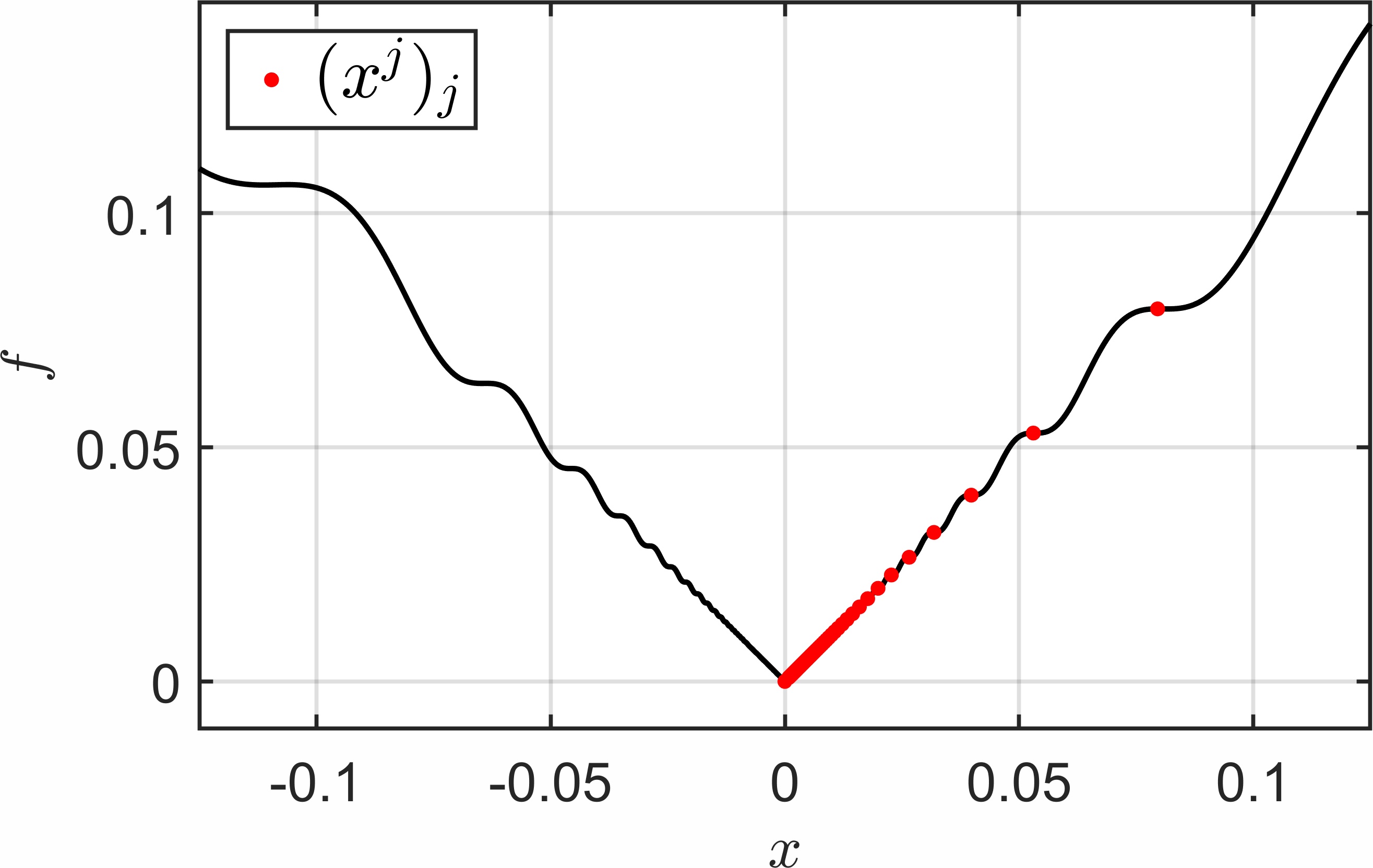}\\
                (a)
    		}
            \parbox[b]{0.49\textwidth}{
                \centering 
                \includegraphics[width=0.4\textwidth]{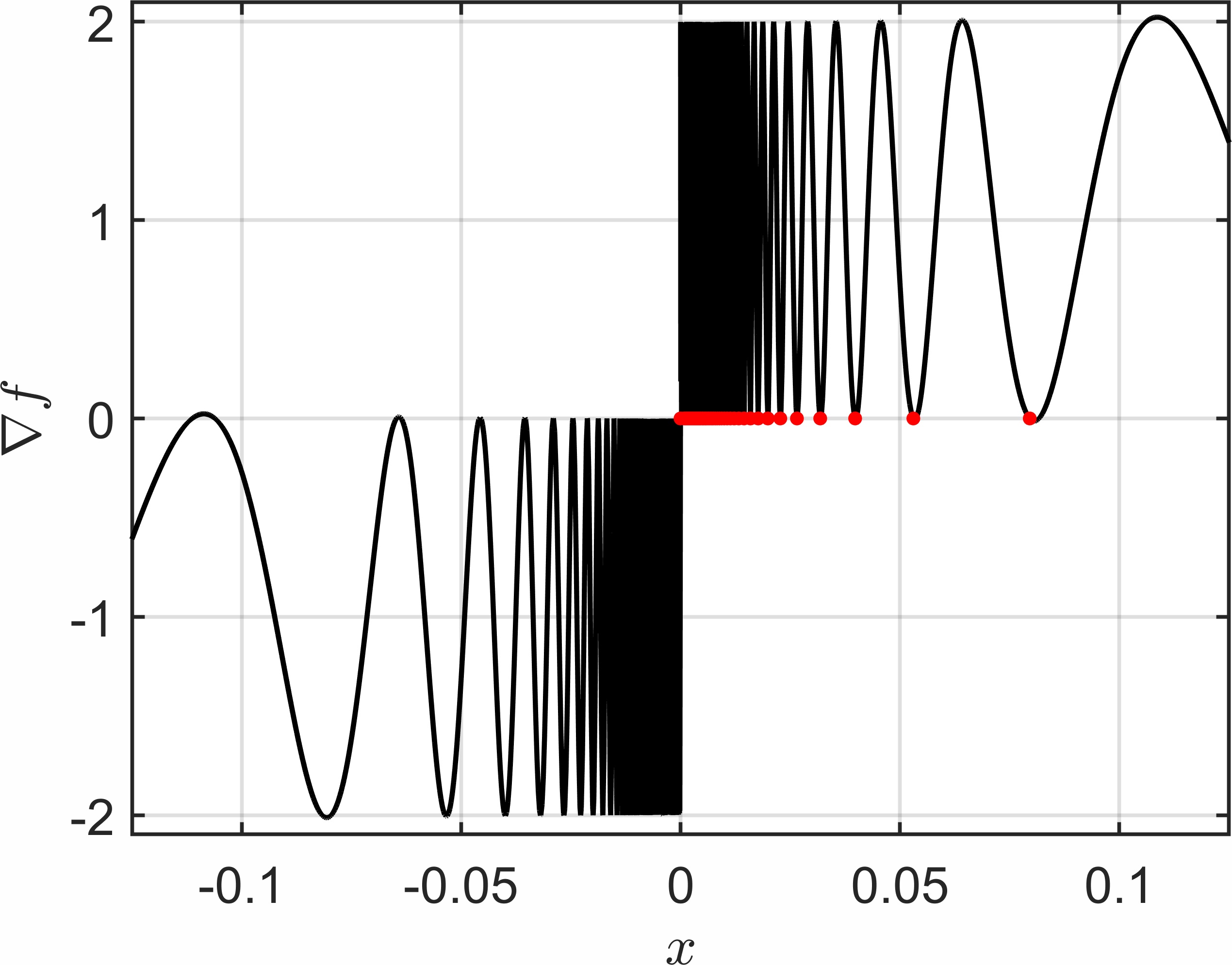}\\
                (b)
    		}
            \caption{The graphs of (a) $f$ and (b) $\nabla f|_{\R \setminus \{ 0 \}}$ and the sequence $(x^j)_j$ in Example \ref{example:why_semismooth}.}
            \label{fig:example_why_semismooth}
        \end{figure}   
        Clearly, $x^* = 0$ is a minimum of order $1$ of $f$ (for $\beta < 1$). Consider the sequences $(x^j)_j$, $(\eps_j)_j$ and $(\delta_j)_j$ given by
        \begin{align*}
            x^j = (2 \pi j)^{-1}, \quad \eps_j = 0, \quad \delta_j = 0 \quad \forall j \in \N.
        \end{align*}
        Since $f$ is continuously differentiable outside of $x^* = 0$, we have
        \begin{align*}
            \partial_{\eps_j} f(x^j)
            &= \{ \nabla f(x^j) \}
            = \left\{ 2 x^j \sin \left( \frac{1}{x^j} \right) - \cos \left( \frac{1}{x^j} \right) + \mathrm{sign}(x^j) \right\} \\
            &= \{ 0 \} \quad \forall j \in \N,
        \end{align*}
        so $\min(\| \partial_{\eps_j} f(x^j) \|) = 0 \leq \delta_j$ for all $j \in \N$, as shown in Figure \ref{fig:example_why_semismooth}(b). However, $(x^j)_j$ again converges sublinearly to $x^*$.
    \end{example}

    Example \ref{example:why_semismooth} highlights another issue we have to address: Even when $x^*$ is a minimum of order $1$, such that $f$ grows linearly around $x^*$, $\min(\| \partial_\eps f(x) \|)$ may be arbitrarily small (or even zero) close to $x^*$. More formally, while the mean value theorem \eqref{eq:mean_value_theorem} implies that for any $x \in U$ (with $U$ as in Definition \ref{def:order_minimum}), there is some $\xi \in \partial f(x^* + s (x - x^*))$ for some $s \in (0,1)$ with
    \begin{align} \label{eq:motiv_mvt}
        0 
        < \beta 
        \leq \frac{f(x) - f(x^*)}{\| x - x^* \|}
        = \frac{\langle \xi, x - x^* \rangle}{\| x - x^* \|}
        \leq \frac{ \| \xi \| \| x - x^* \|}{\| x - x^* \|}
        = \| \xi \|,
    \end{align}
    this does not mean that $\beta$ is a lower bound for $\min(\| \partial f(x) \|)$ around $x^*$. The problem is that in the general case, the above inequality only holds for \emph{some} subgradient at \emph{some} point between $x$ and $x^*$. We will avoid this issue by assuming semismoothness of $f$ (cf. \eqref{eq:def_semismooth}), which implies that for $x^j \rightarrow x^*$ with $\frac{x^j - x^*}{\| x^j - x^* \|} =: d^j \rightarrow d \in \R^n$ and any sequence $(\xi^j)_j$ with $\xi^j \in \partial f(x^j)$, it holds
    \begin{align*}
        0 
        < \beta 
        \leq f'(x^*,d)
        = \lim_{j \rightarrow \infty}  \langle \xi^j, d^j \rangle
        \leq \liminf_{j \rightarrow \infty}  \| \xi^j \| \| d^j \|
        = \liminf_{j \rightarrow \infty}  \| \xi^j \|.
    \end{align*}
    This shows that $\langle \xi^j, d^j \rangle$ and, in particular, $\min(\| \partial f(x^j) \|)$, are bounded below close to $x^*$. We will later show that if $(\eps_j)_j$ decreases quickly in relation to $(\| x^j - x^*\|)_j$, then even for $\xi^j \in \partial f_{\eps_j}(x^j)$, $\langle \xi^j, d^j \rangle$ is still bounded below. This will be a key argument in the proof of Theorem \ref{thm:xj_speed_higher_order} for deriving the relationship between $(x^j)_j$, $(\eps_j)_j$ and $(\delta_j)_j$ in Section \ref{sec:analysis_speed_of_convergence}.

    Unfortunately, while semismoothness is sufficient for avoiding the issue highlighted in the previous example, where the minimum is of order $1$, it is not sufficient for minima of higher orders, as our final example in this section shows:

    \begin{example} \label{example:why_higher_order_semismooth}
        For $p \in \N$ consider the function
        \begin{align*}
            f : \R \rightarrow \R, \quad x \mapsto
            \begin{cases}
                 x^{p+1} \sin \left( \frac{1}{x} \right) + \frac{1}{p} |x|^p, & x \neq 0, \\
                 0, & x = 0.
            \end{cases}
        \end{align*}
        It is easy to see that $x^* = 0$ is a minimum of order $p$. If $p \geq 2$ then $f$ is continuously differentiable (and, in particular, semismooth) with $\nabla f(0) = 0$ and
        \begin{align*}
            \nabla f(x) = x^{p-1} \left( (p+1) x \sin \left( \frac{1}{x} \right) - \cos\left( \frac{1}{x} \right)  + \mathrm{sign}(x) \right) \quad \forall x \in \R \setminus \{ 0 \}.
        \end{align*}
        With the same sequences $(x^j)_j$, $(\eps_j)_j$ and $(\delta_j)_j$ as in Example \ref{example:why_semismooth}, we have $\min(\| \partial_{\eps_j} f(x^j) \|) = 0 \leq \delta_j$ for all $j \in \N$. The graphs of $f$ and $\nabla f$ are shown in Figure \ref{fig:example_why_higher_order_semismooth}(a) and (b), respectively.
        \begin{figure}
            \centering
            \parbox[b]{0.32\textwidth}{
                \centering 
                \includegraphics[width=0.32\textwidth]{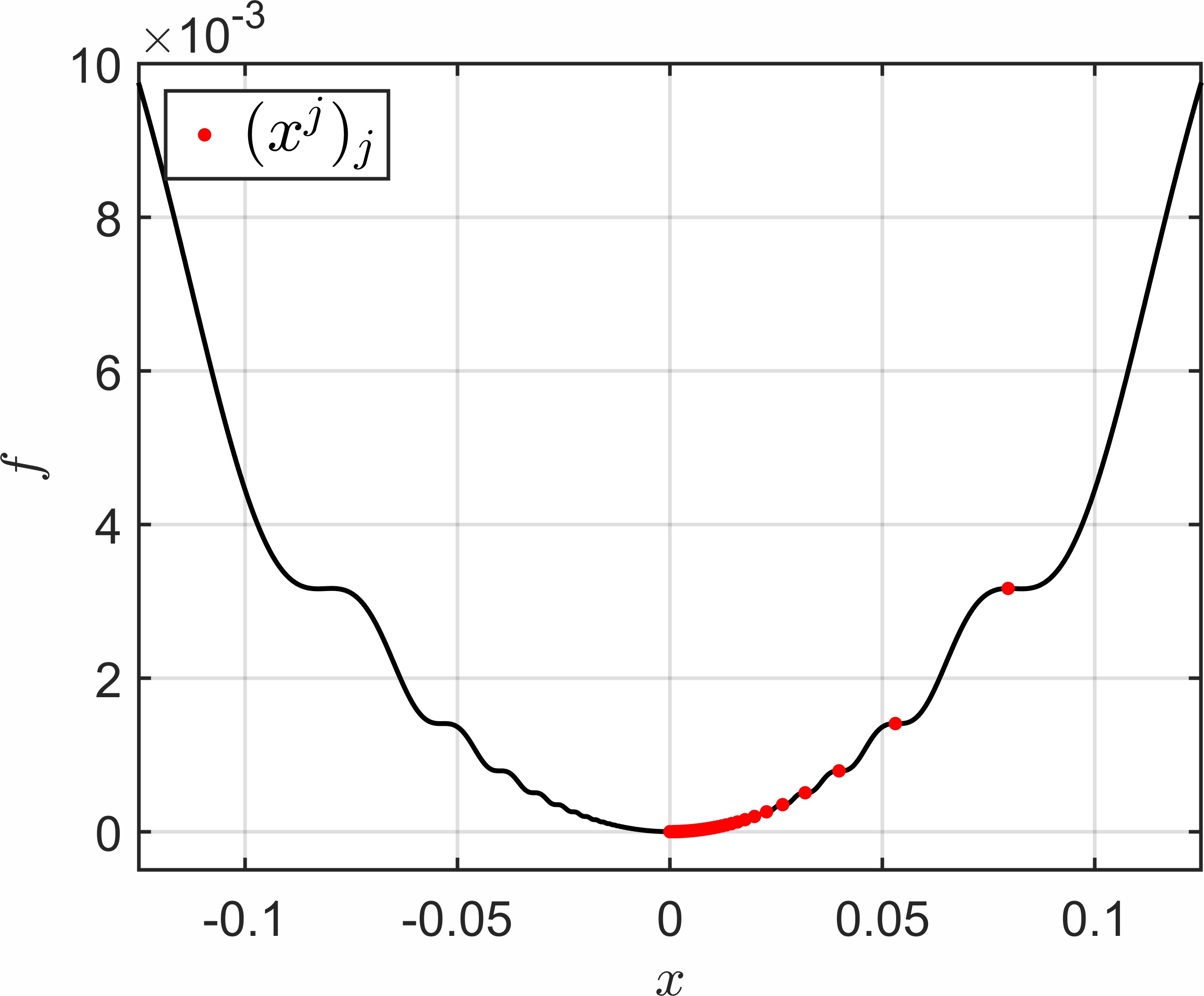}\\
                (a)
    		}
            \parbox[b]{0.32\textwidth}{
                \centering 
                \includegraphics[width=0.32\textwidth]{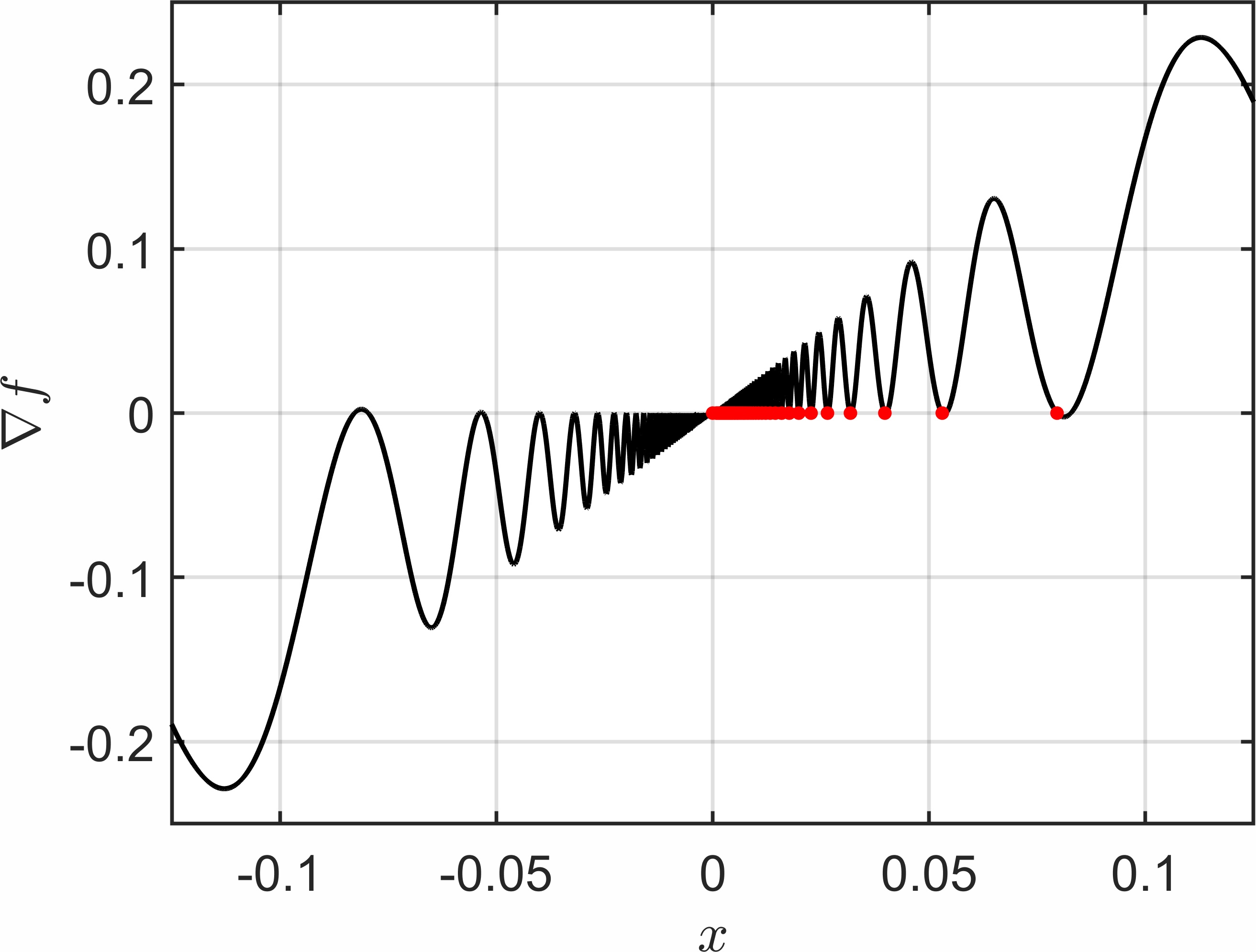}\\
                (b)
    		}
            \parbox[b]{0.32\textwidth}{
                \centering 
                \includegraphics[width=0.31\textwidth]{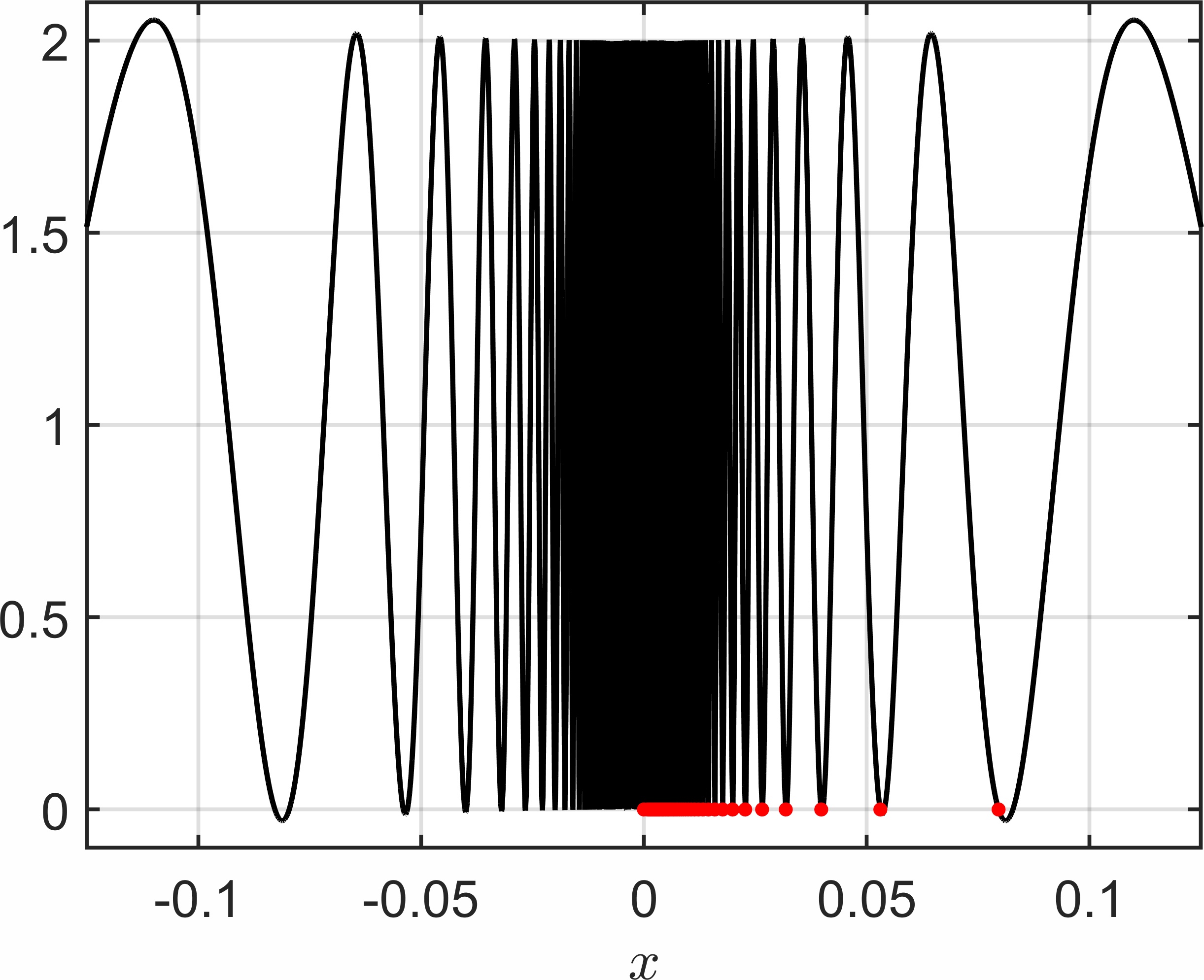}\\
                (c)
    		}
            \caption{(a) The graph of $f$ and the sequence $(x^j)_j$ in Example \ref{example:why_higher_order_semismooth}. (b) The graph of $\nabla f$. (c) The graph of $x \mapsto \nabla f(x) / |x - x^*|$.}
            \label{fig:example_why_higher_order_semismooth}
        \end{figure}   
    \end{example}

    For minima of order $p \geq 2$, we can argue analogously to \eqref{eq:motiv_mvt} to obtain
    \begin{align} \label{eq:motiv_higher_order_semismooth}
        0 
        < \beta 
        \leq \frac{f(x) - f(x^*)}{\| x - x^* \|^p}
        = \frac{\langle \xi, x - x^* \rangle}{\| x - x^* \|^p}
        = \frac{\langle \xi, \frac{x - x^*}{\| x - x^* \|} \rangle}{\| x - x^* \|^{p-1}}
        \leq \frac{ \| \xi \|}{\| x - x^* \|^{p-1}}.
    \end{align}
    If this inequality would hold locally around $x^*$ and for any $\xi \in \partial f(x)$, then $\| \xi \|$ would have at least $p-1$ order growth. Unfortunately, Example \ref{example:why_higher_order_semismooth} shows that semismoothness is not sufficient for this. More precisely, Figure \ref{fig:example_why_higher_order_semismooth}(c) shows that for $x > 0$, we may have
    \begin{align*}
        \frac{\langle \xi^j, \frac{x^j - x^*}{\| x^j - x^* \|} \rangle}{\| x^j - x^* \|^{p-1}}
        = \frac{\nabla f(x^j)}{|x^j|} = 0 \quad \forall j \in \N.
    \end{align*}
    Thus, for minima of order $p \geq 2$, we need a ``higher-order'' version of semismoothness, which we explore in the next section.

\section{Higher-order semismoothness property} \label{sec:higher_order_semismooth}

    In the previous section, we showed the need for a higher-order semismoothness property to be able to derive the speed of convergence of $(x^j)_j$ from $(\eps_j)_j$ and $(\delta_j)_j$. More precisely, motivated by \eqref{eq:motiv_higher_order_semismooth}, for a minimum $x^* \in \R^n$ of order $p \in \N$ with constant $\beta > 0$, we require that
    \begin{align} \label{eq:p_order_min_semismooth}
        \liminf_{\xi \in \partial f(x^* + t d'), t \searrow 0, d' \rightarrow d} \frac{\langle \xi,  d' \rangle}{t^{p-1}} \geq \beta \| d \|^p \quad \forall d \in \R^n \setminus \{ 0 \}.
    \end{align}
    (For ease of notation, compared to \eqref{eq:motiv_higher_order_semismooth}, we substituted $x - x^* = td'$. Furthermore, we consider the limit inferior as an element of $\R \cup \{ -\infty, \infty \}$.) In this section, we show that for $p = 1$, this inequality holds for semismooth functions, and for $p \geq 1$, it holds for piecewise differentiable functions that have a certain convexity property. Unfortunately, general nonconvex piecewise differentiable functions do not satisfy \eqref{eq:p_order_min_semismooth}, which we will demonstrate in an example. 

    Note that for minima of order $p = 1$, the denominator on the left-hand side of \eqref{eq:p_order_min_semismooth} simply becomes $1$, so we immediately obtain the following result:
    \begin{lemma} \label{lem:semismooth_order_1}
        If $f : \R^n \rightarrow \R$ is semismooth and $x^*$ is a minimum of order $p = 1$, then \eqref{eq:p_order_min_semismooth} holds.
    \end{lemma}
    \begin{proof}
        By semismoothness of $f$ we have
        \begin{align*}
            \lim_{\xi \in \partial f(x^* + t d'), t \searrow 0, d' \rightarrow d} \langle \xi, d' \rangle 
            &= f'(x^*,d) 
            = \lim_{t \searrow 0} \frac{f(x^* + t d) - f(x^*)}{t} \\
            &\geq \lim_{t \searrow 0} \frac{\beta \| t d \|}{t}
            = \beta \| d \|
        \end{align*}
        for all $d \in \R^n$.
    \end{proof}

    The proof of Lemma \ref{lem:semismooth_order_1} is based on the observation that for $p = 1$, the right-hand side of \eqref{eq:p_order_min_semismooth} is a lower bound for the directional derivative $f'(x^*,d)$ which, in turn, is equal to the left-hand side due to semismoothness. For $p \geq 2$, we follow a similar strategy. To this end, we employ the following higher-order directional derivative \cite{S1986,GK1992}:

    \begin{definition}
        Let $x \in \R^n$, $d \in \R^n$ and $p \in \N$. Then
        \begin{align*}
            \ldDH^p f(x,d) := \liminf_{t \searrow 0, d' \rightarrow d} \frac{f(x + td') - f(x)}{t^p} \in \R \cup \{ -\infty, \infty \}
        \end{align*}
        is called the $p$\emph{-order lower (Dini-Hadamard) directional derivative} of $f$ at $x$ in the direction $d$.
    \end{definition}

    Clearly, if $x^*$ is a local minimum of $f$, then $\ldDH^p f(x^*,d) \geq 0$ for all $d \in \R^n$, $p \in \N$. Moreover, as in the first-order case, if $x^*$ is a minimum of order $p \in \N$, then the right-hand side of \eqref{eq:p_order_min_semismooth} is a lower bound for $\ldDH^p f(x^*,d)$:
    \begin{lemma} \label{lem:min_order_dir_deriv}
        Let $f : \R^n \rightarrow \R$. If $x^*$ is a minimum of order $p \in \N$ with constant $\beta > 0$, then 
        \begin{align*}
            \ldDH^p f(x^*,d) \geq \beta \| d \|^p \quad \forall d \in \R^n.
        \end{align*}
    \end{lemma}
    \begin{proof}
        It holds 
        \begin{align*}
            \ldDH^p f(x^*,d) 
            &= \liminf_{t \searrow 0, d' \rightarrow d} \frac{f(x^* + td') - f(x^*)}{t^p} 
            \geq \liminf_{t \searrow 0, d' \rightarrow d} \frac{\beta \| t d' \|^p}{t^p} \\
            &= \liminf_{t \searrow 0, d' \rightarrow d} \beta \| d' \|^p
            = \beta \| d \|^p
        \end{align*}
        for all $d \in \R^n$.
    \end{proof}    

    It remains to analyze when 
    \begin{align} \label{eq:higher_order_semismooth}
        \liminf_{\xi \in \partial f(x^* + t d'), t \searrow 0, d' \rightarrow d} \frac{\langle \xi,  d' \rangle}{t^{p-1}}
        \geq \ldDH^p f(x^*,d) \quad \forall d \in \R^n \setminus \{ 0 \}
    \end{align}
    holds (for a minimum $x^*$ of order $p$), which can be interpreted as a form of higher-order semismoothness. (Note that if the left-hand side of this inequality equals $\infty$, then \eqref{eq:p_order_min_semismooth} automatically holds.) Before showing that a certain class of convex, piecewise differentiable functions possesses this property, we briefly discuss the relationship of \eqref{eq:higher_order_semismooth} to an existing concept of higher-order semismoothness:
    
    \begin{remark}
        (a) In \cite{QS1993}, a locally Lipschitz continuous function $f$ is called $q$\emph{-order semismooth} at $x^* \in \R^n$ for $q \in (0,1]$, if it is directionally differentiable at $x^*$ and
        \begin{align*}
            \langle \xi, h \rangle - f'(x^*,h) = O(\| h \|^{1+q}) \quad \text{for} \ h \rightarrow 0, \xi \in \partial f(x + h). 
        \end{align*}
        (For $q = 1$, this notion also appeared in \cite{Q1993}, Lemma 2.3, and was renamed to \emph{strong semismoothness} in \cite{QH1997}.) In light of \eqref{eq:higher_order_semismooth}, for $h = t d'$ this leads to
        \begin{align*}
            \frac{\langle \xi, d' \rangle}{t^{p-1}}
            = \frac{\langle \xi, h \rangle}{t^p}
            = \frac{f'(x^*,d')}{t^{p-1}} +  \frac{O(\| t d' \|^{1+q})}{t^p}
            = \frac{f'(x^*,d')}{t^{p-1}} + \frac{O(t^{1+q})}{t^p}.
        \end{align*}
        Unfortunately, $\ldDH^p f(x^*,d)$ is not a lower bound for the right-hand side of this equality. For example, for $p = 2$, $q = 1$ and $f$ as in Example \ref{example:why_higher_order_semismooth}, it holds $f'(0,h) = 0$ for all $h \in \R$ and
        \begin{align*}
            \langle \xi, h \rangle - f'(0,h)
            &= \langle \nabla f(h), h \rangle \\
            &= h^2 \left( 2 h \sin \left( \frac{1}{h} \right) - \cos\left( \frac{1}{h} \right)  + \mathrm{sign}(h) \right)
            = O(h^2)
        \end{align*}
        for $h \rightarrow 0$, so $f$ is $1$-order semismooth at $x^* = 0$ in the sense of \cite{QS1993}. However, as shown in Example \ref{example:why_higher_order_semismooth}, $f$ does not satisfy \eqref{eq:higher_order_semismooth}.\\
        (b) Considering that both sides of \eqref{eq:higher_order_semismooth} contain the limit \emph{inferior}, a proper name for this property might have to include the prefix \emph{lower} or \emph{upper}, similar to the concept of \emph{weak upper semismoothness} in \cite{M1977b}. However, since this would still not be entirely consistent for $p = 1$, we refrain from introducing a fixed name for property \eqref{eq:higher_order_semismooth} in this article, and instead simply regard it as some form of higher-order semismoothness.
    \end{remark}

    The difficulty of verifying condition \eqref{eq:higher_order_semismooth} comes from the fact that it implicitly contains higher-order derivatives of $f$: If $f$ would be twice continuously differentiable, then for $p = 2$, Taylor expansion of $f$ and $\nabla f$ in the minimum $x^*$ would yield
    \begin{align*}
        &\ldDH^p f(x^*,d) = \liminf_{t \searrow 0, d' \rightarrow d} \frac{f(x^* + td') - f(x^*)}{t^2} \\
        &= \liminf_{t \searrow 0, d' \rightarrow d} \frac{f(x^*) + t \langle \nabla f(x^*), d' \rangle + \frac{1}{2} t^2 d'^\top \nabla^2 f(x^*) d' + o(\| t d' \|^2) - f(x^*)}{t^2} \\
        &= \liminf_{t \searrow 0, d' \rightarrow d} \frac{1}{2} d'^\top \nabla^2 f(x^*) d' + \frac{o(\| t d' \|^2)}{t^2} 
        = \frac{1}{2} d^\top \nabla^2 f(x^*) d
    \end{align*}
    and
    \begin{align*}
        \frac{\langle \xi, d' \rangle}{t} 
        &= \frac{\langle \nabla f(x^*+  t d'), d' \rangle}{t} 
        = \frac{\langle \nabla f(x^*) + \nabla^2 f(x^*)(t d') + o(\| t d' \|), d' \rangle}{t} \\
        &= d'^\top \nabla^2 f(x^*) d' + \frac{o(\| t d' \|)}{t} d' 
        \rightarrow d^\top \nabla^2 f(x^*) d \quad \text{as} \quad t \searrow 0, d' \rightarrow d.
    \end{align*}
    So for a twice continuously differentiable function, \eqref{eq:higher_order_semismooth} holds if the Hessian in $x^*$ is positive semidefinite. While this simple relationship is lost when $f$ is nonsmooth, we can still rely on some form of higher-order derivatives when restricting ourselves to piecewise differentiable functions. More precisely, as above, we will compute Taylor expansions of all smooth pieces separately and on both sides of \eqref{eq:higher_order_semismooth} and then show that the piecewise nature of $f$ does not cause any issues.

    To state and prove the main result of this section, we first need some additional notation. For a $p$-times continuously differentiable function $f$ and $k \in \{1,\dots,p\}$, we denote
    \begin{align*}
        &\deriv^{(k)} f(x)(d)^k 
        := \sum_{i_1 = 1}^n \dots \sum_{i_k = 1}^n \partial_{i_1} \dots \partial_{i_k} f(x) d_{i_1} \dots d_{i_k}, \\
        &T_p f(y,x) := \sum_{k = 0}^p \frac{1}{k!} \deriv^{(k)} f(x) (y - x)^k.
    \end{align*}
    Then Taylor expansion of $f$ (see, e.g., \cite{K2004}, p. 66) yields $f(y) = T_p f(y,x) + o(\| y - x \|^p)$. Furthermore, for a piecewise differentiable function $f$ (cf. Section \ref{sec:basics}) with selection functions $f_1, \dots, f_m$, let
    \begin{align*}
        C_i(x) := \{ d \in \R^n : \ &\exists (d^j)_j \in \R^n, (t_j)_j \in \R^{> 0} \text{ with } d^j \rightarrow d, t_j \rightarrow 0\\
        &\text{ and } i \in A(x + t_j d^j) \ \forall j \in \N\}
    \end{align*}
    for $i \in \{1,\dots,m\}$. In words, $C_i(x)$ is the cone of directions at $x$ in which $f$ admits the value of $f_i$.

    \begin{lemma} \label{lem:convex_higher_order}
        Let $x^*$ be a minimum of order $p \in \N$ with constant $\beta > 0$ and $f$ be piecewise $p$-times continuously differentiable with selection functions $f_1, \dots, f_m$, $m \in \N$. If for every $i \in \{1,\dots,m\}$ there is an open set $V_i \subseteq \R^n$ with $C_i(x^*) \setminus \{ 0 \} \subseteq V_i$ and
        \begin{align} \label{eq:higher_order_pos_def}
            \deriv^{(k)} f_i(x^*)(d)^k \geq 0 \quad \forall k \in \{2,\dots,p\}, d \in V_i,
        \end{align}
        then \eqref{eq:higher_order_semismooth} and, in particular, \eqref{eq:p_order_min_semismooth} holds.
    \end{lemma}
    \begin{proof}
        We begin by choosing explicit sequences for the limit inferior in \eqref{eq:higher_order_semismooth}:
        Let $(d^j)_j \in \R^n$, $(t_j)_j \in \R^{> 0}$ and $(\xi^j)_j \in \R^n$ with $t_j \rightarrow 0$, $d^j \rightarrow d \neq 0$ and $\xi^j \in \partial f(x^* + t_j d^j)$ for all $j \in \N$. For ease of notation, let $x^j := x^* + t_j d^j$. For $U \subseteq \R^n$ as in Definition \ref{def:order_minimum}, assume w.l.o.g. that $x^j \in U$ for all $j \in \N$.\\
        \textbf{Step 1:} Assume w.l.o.g. that any $i \in A(x^j)$ for any $j \in \N$ is active infinitely many times along $(x^j)_j$, and let $I \subseteq \{1,\dots,m\}$ be the set of such indices. Then $I \subseteq A(x^*)$ (due to continuity of $f$) and by definition of $C_i(x^*)$, we have
        \begin{align*}
            d \in \bigcap_{i \in I} ( C_i(x^*) \setminus \{ 0 \} ) \subseteq \bigcap_{i \in I} V_i.
        \end{align*}
        Since the right-hand side is open and $d^j \rightarrow d$, we can assume w.l.o.g. that $d^j \in V_i$ for all $i \in I$, $j \in \N$. \\
        \textbf{Step 2:} For $j \in \N$ and $i \in \{1,\dots,m\}$ let
        \begin{align*}
            \varphi : \R^n \rightarrow \R, \quad x \mapsto \langle \nabla f_i(x), d^j \rangle.
        \end{align*}
        Then $\varphi$ is $(p-1)$-times continuously differentiable by assumption. Taylor expansion of $\varphi$ at $x^*$, combined with the identity
        \begin{align*}
            \deriv^{(k)} \varphi(x^*)(t_j d^j)^k = t_j^{k} \deriv^{(k+1)} f_i(x^*)(d^j)^{k+1} \quad \forall k \in \{1,\dots,p-1\},
        \end{align*}
        yields
        \begin{equation} \label{eq:proof_lem_PC_slope_growth_convex_2}
            \begin{aligned}
                \frac{\langle \nabla f_i(x^j), d^j \rangle}{t_j^{p-1}}
                &= t_j^{-(p-1)} \varphi(x^* + t_j d^j) \\
                &= t_j^{-(p-1)} T_{p-1} \varphi(x^* + t_j d^j,x^*) + \frac{o(\| t_j d^j \|^{p-1})}{t_j^{p-1}} \\
                &= t_j^{-(p-1)} \sum_{k = 0}^{p-1} \frac{1}{k!} \deriv^{(k)} \varphi(x^*)(t_j d^j)^k + \frac{o(\| t_j d^j \|^{p-1})}{t_j^{p-1}} \\
                &= t_j^{-(p-1)} \sum_{k = 0}^{p-1} \frac{1}{k!} t_j^{k} \deriv^{(k+1)} f_i(x^*)(d^j)^{k+1} + \frac{o(\| t_j d^j \|^{p-1})}{t_j^{p-1}} \\
                &= t_j^{-(p-1)} \sum_{k = 1}^{p} \frac{1}{(k-1)!} t_j^{k-1} \deriv^{(k)} f_i(x^*)(d^j)^{k} + \frac{o(\| t_j d^j \|^{p-1})}{t_j^{p-1}} \\
                &= \sum_{k = 1}^{p} \frac{1}{(k-1)!} t_j^{k-p} \deriv^{(k)} f_i(x^*)(d^j)^{k} + \frac{o(\| t_j d^j \|^{p-1})}{t_j^{p-1}} 
            \end{aligned}
        \end{equation}
        for all $j \in \N$.\\
        \textbf{Step 3:} Let $i \in \{1,\dots,m\}$. Taylor expansion of $f_i$ at $x^*$ yields
        \begin{equation} \label{eq:proof_lem_PC_slope_growth_convex_1}
            \begin{aligned}
                \frac{f_i(x^j) - f_i(x^*)}{t_j^p}
                &= t_j^{-p} ( T_{p} f_i(x^* + t_j d^j,x^*) - f_i(x^*) ) + \frac{o(\| t_j d^j \|^{p})}{t_j^p} \\
                &= t_j^{-p} \sum_{k = 1}^{p} \frac{1}{k!} \deriv^k f_i(x^*)(t_j d^j)^k + \frac{o(\| t_j d^j \|^{p})}{t_j^p} \\
                &= \sum_{k = 1}^{p} \frac{1}{k!} t_j^{k-p} \deriv^k f_i(x^*)(d^j)^k + \frac{o(\| t_j d^j \|^{p})}{t_j^p}
            \end{aligned}
        \end{equation}
        for all $j \in \N$. \\
        \textbf{Step 4:} Let $i \in I$ (cf. Step 1). When comparing the sums on the right-hand sides of \eqref{eq:proof_lem_PC_slope_growth_convex_2} and \eqref{eq:proof_lem_PC_slope_growth_convex_1}, we see that they only differ by the coefficients $1/k!$ and $1/(k-1)!$. Since $1/k! < 1/(k-1)!$ for $k \geq 2$ and equality holds for $k = 1$ (since $1! = 0! = 1$), assumption \eqref{eq:higher_order_pos_def} and Step 1 allow us to estimate
        \begin{align} \label{eq:proof_lem_PC_slope_growth_convex_5}
            \sum_{k = 1}^{p} \frac{1}{k!} t_j^{k-p} \deriv^k f_i(x^*)(d^j)^k
            \leq \sum_{k = 1}^{p} \frac{1}{(k-1)!} t_j^{k-p} \deriv^{(k)} f_i(x^*)(d^j)^{k},
        \end{align}
        and, in particular,
        \begin{equation}  \label{eq:proof_lem_PC_slope_growth_convex_3}
            \begin{aligned}
                \liminf_{j \rightarrow \infty} \frac{f_i(x^j) - f_i(x^*)}{t_j^p}
                &\leq \liminf_{j \rightarrow \infty} \sum_{k = 1}^{p} \frac{1}{(k-1)!} t_j^{k-p} \deriv^{(k)} f_i(x^*)(d^j)^{k} \\
                &= \liminf_{j \rightarrow \infty} \frac{\langle \nabla f_i(x^j), d^j \rangle}{t_j^{p-1}}.
            \end{aligned}
        \end{equation}
        \textbf{Step 5:} For $i \in I$ let $J(i) := \{ j \in \N : i \in A(x^j) \}$. By Step 1 $J(i)$ is unbounded. As in \eqref{eq:proof_lem_PC_slope_growth_convex_3} we can estimate
        \begin{equation}
            \begin{aligned} \label{eq:proof_lem_PC_slope_growth_convex_4}
                \ldDH^p f(x^*,d) 
                &= \liminf_{t \searrow 0, d' \rightarrow d} \frac{f(x^* + t d') - f(x^*)}{t^p}
                \leq \liminf_{j \in J(i), j \rightarrow \infty} \frac{f(x^* + t_j d^j) - f(x^*)}{t_j^p} \\
                &= \liminf_{j \in J(i), j \rightarrow \infty} \frac{f_i(x^j) - f_i(x^*)}{t_j^p}
                \leq \liminf_{j \in J(i), j \rightarrow \infty} \frac{\langle \nabla f_i(x^j), d^j \rangle}{t_j^{p-1}}.
            \end{aligned}
        \end{equation}
        \textbf{Step 6:} Since $f$ is piecewise differentiable, it holds
        \begin{align*}
            \xi^j \in \partial f(x^j) \subseteq \conv(\{ \nabla f_i(x^j) : i \in A(x^j) \}) \quad \forall j \in \N.
        \end{align*}
        Denote $I = \{ i_1, \dots, i_{|I|} \}$ (with $I$ from Step 1). Then there is a sequence $(\alpha^j)_j \in \R^{|I|}$ with $\alpha_l^j \geq 0$ for all $l \in \{1, \dots, |I| \}$, $\xi^j = \sum_{l = 1}^{|I|} \alpha_l^j \nabla f_{i_l}(x^j)$ for all $j \in \N$ and $\alpha_l^j = 0$ whenever $j \notin J(i_l)$ (i.e., whenever $f_{i_l}$ is inactive at $x^j$). For $j \in \N$ we obtain
        \begin{align*}
            \frac{\langle \xi^j, d^j \rangle}{t_j^{p-1}}
            &= \sum_{l = 1}^{|I|} \alpha_l^j \frac{\langle \nabla f_{i_l}(x^j), d^j \rangle}{t_j^{p-1}}
            = \sum_{l : j \in J(i_l)} \alpha_l^j \frac{\langle \nabla f_{i_l}(x^j), d^j \rangle}{t_j^{p-1}} \\
            &\geq \min_{l : j \in J(i_l)} \frac{\langle \nabla f_{i_l}(x^j), d^j \rangle}{t_j^{p-1}}.
        \end{align*}
        Since the minimum on the right-hand side in the previous inequality is taken over a finite set, using \eqref{eq:proof_lem_PC_slope_growth_convex_4} yields
        \begin{align*}
            \liminf_{j \rightarrow \infty} \frac{\langle \xi^j, d^j \rangle}{t_j^{p-1}}
            &\geq \liminf_{j \rightarrow \infty} \min_{l : j \in J(i_l)} \frac{\langle \nabla f_{i_l}(x^j), d^j \rangle}{t_j^{p-1}} \\
            &= \min_{l \in \{1,\dots,|I|\}} \liminf_{j \in J(i_l), j \rightarrow \infty} \frac{\langle \nabla f_{i_l}(x^j), d^j \rangle}{t_j^{p-1}} 
            \stackrel{\eqref{eq:proof_lem_PC_slope_growth_convex_4}}{\geq} \ldDH^p f(x^*,d).
        \end{align*}
        Since $(\xi^j)_j$, $(d^j)_j$ and $(t_j)_j$ were chosen as arbitrary sequences in the limit inferior on the left-hand side of \eqref{eq:higher_order_semismooth}, this shows that \eqref{eq:higher_order_semismooth} holds.
    \end{proof}

    For the special case $p = 2$, we obtain the following corollary:

    \begin{corollary} \label{cor:p_eq_2}
        Let $f : \R^n \rightarrow \R$, $x \mapsto \max_{i \in \{1,\dots,m\}} f_i(x)$, for twice continuously differentiable, strongly convex functions $f_1, \dots, f_m$, $m \in \N$. Then $f$ has a unique minimum of order $2$ that satisfies \eqref{eq:higher_order_semismooth} and, in particular, \eqref{eq:p_order_min_semismooth}.
    \end{corollary}
    \begin{proof}
        Since strong convexity implies strict convexity and since the maximum of strictly convex functions is strictly convex, $f$ has a unique minimum $x^* \in \R^n$. By strong convexity of $f_1, \dots, f_m$, there is some $\eta > 0$ with $d^\top \nabla^2 f_i(x^*) d \geq \eta \| d \|^2$ for all $d \in \R^n$, $i \in \{1,\dots,m\}$. Thus, \eqref{eq:higher_order_pos_def} holds for $p = 2$ (with $V_i = \R^n$, $i \in \{1,\dots,m\}$). Since $f$ is piecewise differentiable, there is some $\alpha \in \R^m$ with $\alpha_i \geq 0$ for all $i \in \{1,\dots,m\}$, $\alpha_i = 0$ whenever $i \notin A(x^*)$, $\sum_{i = 1}^m \alpha_i = 1$ and 
        \begin{align*}
            \sum_{i = 1}^m \alpha_i \nabla f_i(x^*) = 0. 
        \end{align*}
        Now Taylor expansion of $\psi(x) := \sum_{i = 1}^m \alpha_i f_i(x)$ at $x^*$ yields
        \begin{align*}
            f(x) 
            &= \max_{i \in \{1,\dots,m\}} f_i(x)
            \geq \psi(x) \\
            &= \psi(x^*) + \nabla \psi(x^*) + \frac{1}{2} (x - x^*)^\top \nabla^2 \psi(x^*) (x - x^*) + o(\| x - x^* \|^2) \\
            &= f(x^*) + \frac{1}{2} (x - x^*)^\top \left( \sum_{i = 1}^m \alpha_i \nabla^2 f_i(x^*) \right) (x - x^*) + o(\| x - x^* \|^2) \\
            &\geq f(x^*) + \frac{1}{2} \eta \| x - x^* \|^2 + o(\| x - x^* \|^2) \\
            &= f(x^*) + \frac{1}{2} \left( \eta + \frac{o(\| x - x^* \|^2)}{\| x - x^* \|^2} \right) \| x - x^* \|^2 \quad \forall x \in \R^n.
        \end{align*}
        Thus $x^*$ is a minimum of order $2$ with a constant $\beta < \frac{1}{2} \eta$. Application of Lemma \ref{lem:convex_higher_order} completes the proof.    
    \end{proof}

    Example \ref{example:higher_order_convex} in the appendix shows a case where the minimum is of order $3$. Furthermore, it shows that the inequality \eqref{eq:higher_order_pos_def} that we required for Lemma \ref{lem:convex_higher_order} does not imply that the selection functions have to be convex.

    Unfortunately, properties \eqref{eq:p_order_min_semismooth} and \eqref{eq:higher_order_semismooth} may fail to hold for nonconvex, piecewise differentiable functions when $p \geq 2$. The reason for this is the fact that when $f$ is nonconvex, the higher-order terms $\deriv^k f_i(x^*)(d^j)^k$ in Step 4 in the proof of Lemma \ref{lem:convex_higher_order} may be negative, such that the estimate \eqref{eq:proof_lem_PC_slope_growth_convex_5} does not hold. This is demonstrated in Example \ref{example:crescent} in the appendix. 

\section{Deriving the speed of convergence} 
\label{sec:analysis_speed_of_convergence}

    In the previous section, we showed that property \eqref{eq:p_order_min_semismooth} is implied by the higher-order semismoothness property \eqref{eq:higher_order_semismooth} which, for minima of order $1$, holds for semismooth functions, and for minima of order larger than $1$, holds for a class of convex, piecewise differentiable functions.
    In the following theorem, which we regard as the main result of this article, we show how property \eqref{eq:p_order_min_semismooth} can be used to derive a relationship between the speed of convergence of $(x^j)_j$ and the speeds of $(\eps_j)_j$ and $(\delta_j)_j$:

    \begin{theorem} \label{thm:xj_speed_higher_order}
        Let $f : \R^n \rightarrow \R$ be locally Lipschitz continuous. Let $(x^j)_j \in \R^n$, $(\eps_j)_j \in \R^{\geq 0}$ and $(\delta_j)_j \in \R^{\geq 0}$ such that $x^j \rightarrow x^* \in \R^n$, $\eps_j \rightarrow 0$ and $\min(\| \partial_{\eps_j} f(x^j) \|) \leq \delta_j$ for all $j \in \N$.
        Assume that $x^*$ is a minimum of order $p \in \N$ with constant $\beta > 0$ and that \eqref{eq:p_order_min_semismooth} holds in $x^*$.
        \begin{enumerate}
            \item[(i)] If $p = 1$ and $\delta_j \rightarrow \Bar{\delta} < \beta$, then there are $M > 0$ and $N \in \N$ such that
            \begin{align} \label{eq:xj_speed_higher_order_1}
                \| x^j - x^* \| \leq M \eps_j \quad \forall j > N.
            \end{align}
            \item[(ii)] If $p \geq 2$ then there are $M > 0$ and $N \in \N$ such that
            \begin{align} \label{eq:xj_speed_higher_order_2}
                \| x^j - x^* \| \leq M \max(\eps_j^{\frac{1}{p}}, \delta_j^{\frac{1}{p-1}}) \quad \forall j > N.
            \end{align}
        \end{enumerate}
    \end{theorem}
    \begin{proof}
        Assume that (i) or (ii) do not hold. Then in both cases there is a subsequence $(j_l)_l \in \N$ with $x^{j_l} \neq x^*$ for all $l \in \N$,
        \begin{align} \label{eq:proof_thm_xj_speed_higher_order_1}
            \lim_{l \rightarrow \infty} \frac{\eps_{j_l}}{\| x^{j_l} - x^* \|^{p}} = 0
            \quad \text{and} \quad
             \lim_{l \rightarrow \infty} \frac{\delta_{j_l}}{\| x^{j_l} - x^* \|^{p-1}} < \beta.
        \end{align}
        (In case (ii) does not hold, the second limit actually vanishes.) \\
        \textbf{Step 1:} For $l \in \N$ let $g^l \in \partial_{\eps_{j_l}} f(x^{j_l})$ with $\| g^l \| \leq \delta_{j_l}$. By Carathéodory's theorem, for $i \in \{1,\dots,n+1\}$, there are $y_i^l \in \Bcl_{\eps_{j_l}}(x^{j_l})$, $\xi_i^l \in \partial f(y_i^l) \subseteq \partial_{\eps_{j_l}} f(x^{j_l})$ and $\alpha_i^l \geq 0$ with $\sum_{i = 1}^{n+1} \alpha_i^l = 1$ and
        \begin{align*}
            g^l = \sum_{i = 1}^{n+1} \alpha_i^l \xi_i^l.
        \end{align*}
        Since $\| y_i^l - x^{j_l} \| \leq \eps_{j_l}$ it holds $\lim_{l \rightarrow \infty} y_i^l = x^*$ for all $i \in \{1,\dots,n+1\}$. By compactness of the $\eps$-subdifferential at $x^*$ (cf. \cite{G1977}, Proposition 2.3), we can assume w.l.o.g. that $\lim_{l \rightarrow \infty} \xi_i^l = \Bar{\xi}_i \in \R^n$ for all $i \in \{1,\dots,n+1\}$. By \cite{C1990}, Proposition 2.1.5, it follows that $\Bar{\xi}_i \in \partial f(x^*)$. Furthermore, by compactness of the set of convex coefficients, we can assume w.l.o.g. that $\lim_{l \rightarrow \infty} \alpha_i^l = \Bar{\alpha}_i$ for all $i \in \{1,\dots,n+1\}$, such that
        \begin{align*}
            \Bar{g} := \sum_{i = 1}^{n+1} \Bar{\alpha}_i \Bar{\xi}_i \in \partial f(x^*)
        \end{align*}
        is the limit of $(g^l)_l$. (In case (ii) we have $\Bar{g} = 0$.) \\
        \textbf{Step 2:} Consider the sequences $(d^l)_l$ and $(d_i^l)_l$ given by
        \begin{align*}
            d^l := \frac{x^{j_l} - x^*}{\| x^{j_l} - x^* \|} \quad \text{and} \quad d_i^l := \frac{y_i^l - x^*}{\| x^{j_l} - x^* \|}, \ i \in \{1,\dots,n+1\}.
        \end{align*}
        By compactness, $(d^l)_l$ must have an accumulation point $\Bar{d} \in \R^n$ with $\| \Bar{d} \| = 1$. Assume w.l.o.g. that $\Bar{d}$ is the limit of $(d^l)_l$. 
        For all $i \in \{1,\dots,n+1\}$, it holds
        \begin{align*}
            \| d_i^l - d^l \| 
            = \frac{\| y_i^l - x^* - (x^{j_l} - x^*) \|}{\| x^{j_l} - x^* \|}
            = \frac{\| y_i^l - x^{j_l} \|}{\| x^{j_l} - x^* \|}
            \leq \frac{\eps_{j_l}}{\| x^{j_l} - x^* \|},
        \end{align*}
        and combined with \eqref{eq:proof_thm_xj_speed_higher_order_1}, we obtain
        \begin{align} \label{eq:proof_thm_xj_speed_higher_order_2}
            \frac{\| d_i^l - d^l \|}{\| x^{j_l} - x^* \|^{p-1}} \leq \frac{\eps_{j_l}}{\| x^{j_l} - x^* \|^{p}} \xrightarrow[]{l \rightarrow \infty} 0.
        \end{align}
        In particular, $\lim_{l \rightarrow \infty} d_i^l = \bar{d}$ for all $i \in \{1,\dots,n+1\}$.\\ 
        \textbf{Step 3:} By construction we have 
        \begin{align*}
            \xi_i^l \in \partial f(y_i^l) 
            = \partial f \left( x^* + \| x^{j_l} - x^* \| \frac{y_i^l - x^*}{\| x^{j_l} - x^* \|} \right) 
            = \partial f(x^* + \| x^{j_l} - x^* \| d_i^l).
        \end{align*}
        Assume w.l.o.g. that $x^{j_l} \in U$ with $U$ as in Definition \ref{def:order_minimum}. Combination of \eqref{eq:p_order_min_semismooth} and \eqref{eq:proof_thm_xj_speed_higher_order_2} (and boundedness of $(\xi_i^l)_l$) yields
        \begin{equation}
            \begin{aligned} \label{eq:proof_thm_xj_speed_higher_order_3}
                \liminf_{l \rightarrow \infty} \frac{\langle \xi_i^l, d^l \rangle}{\| x^{j_l} - x^* \|^{p-1}}
                &= \liminf_{l \rightarrow \infty} \left( \frac{\langle \xi_i^l, d_i^l \rangle}{\| x^{j_l} - x^* \|^{p-1}} + \frac{\langle \xi_i^l, d^l - d_i^l \rangle}{\| x^{j_l} - x^* \|^{p-1}} \right)\\
                &\stackrel{\eqref{eq:proof_thm_xj_speed_higher_order_2}}{=} \liminf_{l \rightarrow \infty} \frac{\langle \xi_i^l, d_i^l \rangle}{\| x^{j_l} - x^* \|^{p-1}} 
                \stackrel{\eqref{eq:p_order_min_semismooth}}{\geq} \beta \| \Bar{d} \|^p = \beta
            \end{aligned}
        \end{equation}
        for all $i \in \{1,\dots,n+1\}$. Furthermore note that
        \begin{equation} \label{eq:proof_thm_xj_speed_higher_order_5}
            \begin{aligned}
                \delta_{j_l}
                &\geq \| g^l \| 
                = \left\| \sum_{i = 1}^{n+1} \alpha_i^l \xi_i^l \right\|
                = \left\| \sum_{i = 1}^{n+1} \alpha_i^l \xi_i^l \right\| \| d^l \| \\
                &\geq \langle \sum_{i = 1}^{n+1} \alpha_i^l \xi_i^l, d^l \rangle
                = \sum_{i = 1}^{n+1} \alpha_i^l \langle \xi_i^l, d^l \rangle \quad \forall l \in \N.
            \end{aligned}
        \end{equation}
        Division of \eqref{eq:proof_thm_xj_speed_higher_order_5} by $\| x^{j_l} - x^* \|^{p-1}$ and combination with \eqref{eq:proof_thm_xj_speed_higher_order_1} and \eqref{eq:proof_thm_xj_speed_higher_order_3} yields
        \begin{equation*}
            \begin{aligned}
                \beta 
                &\stackrel{\eqref{eq:proof_thm_xj_speed_higher_order_1}}{>} \liminf_{l \rightarrow \infty} \frac{\delta_{j_l}}{\| x^{j_l} - x^* \|^{p-1}}
                \stackrel{\eqref{eq:proof_thm_xj_speed_higher_order_5}}{\geq} \liminf_{l \rightarrow \infty} \frac{ \sum_{i = 1}^{n+1} \alpha_i^l \langle \xi_i^l, d^l \rangle}{\| x^{j_l} - x^* \|^{p-1}} \\
                &= \liminf_{l \rightarrow \infty} \sum_{i = 1}^{n+1} \alpha_i^l \frac{\langle \xi_i^l, d^l \rangle}{\| x^{j_l} - x^* \|^{p-1}}
                \geq \sum_{i = 1}^{n+1} \liminf_{l \rightarrow \infty} \alpha_i^l \frac{\langle \xi_i^l, d^l \rangle}{\| x^{j_l} - x^* \|^{p-1}} \\
                &\stackrel{\eqref{eq:proof_thm_xj_speed_higher_order_3}}{\geq} \sum_{i = 1}^{n+1} \Bar{\alpha}_i \beta = \beta,
            \end{aligned}
        \end{equation*}
        which is a contradiction.
    \end{proof}

    \begin{remark} \label{remark:conv_exponent} 
        Let $(a_j)_j \in \R^{\geq 0}$ be a sequence with $a_j \rightarrow 0$ and let $b > 0$. If $(a_j)_j$ Q-converges with order $q \in \N$ and rate $\mu \geq 0$, then for the sequence $(a_j^b)_j$ we have
        \begin{align*}
            \limsup_{j \rightarrow \infty} \frac{| a_{j+1}^b - 0 |}{| a_j^b - 0 |^q}
            = \limsup_{j \rightarrow \infty} \left( \frac{| a_{j+1} - 0 |}{| a_j - 0 |^q} \right)^b
            \leq \mu^b,
        \end{align*}
        so $(a_j^b)_j$ Q-converges with order $q$ and rate $\mu^b$. Furthermore, if $(a_j)_j$ R-converges with order $q$ and rate $\mu$, then R-convergence of $(a^b_j)_j$ with order $q$ and rate $\mu^b$ follows by the same argument (applied to the dominating sequence). Finally, if $(a_j)_j$ converges Q- or R-superlineary, then the same holds for $(a_j^b)_j$, respectively.
    \end{remark}

    By Theorem \ref{thm:xj_speed_higher_order}, $(x^j)_j$ converges at least as fast as the slowest of the two sequences $(\eps_j^{1/p})_j$ and $(\delta_j^{1/(p-1)})_j$. More precisely, considering Remark \ref{remark:conv_exponent}, the order of R-convergence of $(x^j)_j$ is the worst of the orders of R-convergence of $(\eps_j)_j$ and $(\delta_j)_j$, and the rate is either the $p$ or $(p-1)$-th root of the corresponding rate. In Example \ref{example:xj_speed_estim_equal} in the appendix, we show that these estimates are, in a sense, tight.

\section{Application to descent methods} \label{sec:descent_methods}

    In this section we use Theorem \ref{thm:xj_speed_higher_order} to analyze the behavior of a class of descent methods for nonsmooth optimization. Here the sequences $(\eps_j)_j$ and $(\delta_j)_j$ occur as parameters of an algorithm that generates a sequence $(x^j)_j$ with $\min(\| \partial_{\eps_j} f(x^j) \|) \leq \delta_j$. In addition to the analysis of the algorithm, this also gives us an opportunity to showcase Theorem \ref{thm:xj_speed_higher_order} in numerical examples. Note that in the context of this article, we are less interested in the question \emph{if} the algorithm converges, and more interested in the question of \emph{how fast} it converges if it does. As such, convergence of $(x^j)_j$ to a point $x^*$ satisfying the requirements from the previous section is part of our assumptions, and we do not proof the convergence to such points. For example, convergence could be assured by assuming that the initial point is in a sublevel set in which $x^*$ is the only critical point (cf. Section \ref{sec:basics}). In light of this, assuming convergence is a relatively weak assumption.
    We first introduce an abstract version of the descent method, in which no particular approximation of the $\eps$-subdifferential and no particular line search is chosen. For this abstract method we derive a simple convergence result from Theorem \ref{thm:xj_speed_higher_order}. Afterwards, we use the implementable descent method described in \cite{GP2021,G2022,G2024} to perform numerical experiments with common test functions.
    
    \subsection{Abstract descent method}

    For $x \in \R^n$ and $\eps > 0$ consider the element
    \begin{align*}
        \Bar{v} := \argmin_{\xi \in - \partial_\eps f(x)} \| \xi \|^2.
    \end{align*}
    Using convex analysis (\cite{HU1993}, Theorem 3.1.1), we see that either $\Bar{v} = 0$ (i.e., $x$ is $(\eps,0)$-critical) or it holds
    \begin{align*}
        \langle \xi, \Bar{v} \rangle \leq - \| \Bar{v} \|^2 < 0 \quad \forall \xi \in \partial_\eps f(x).
    \end{align*}
    In the latter case, application of the mean value theorem \eqref{eq:mean_value_theorem} shows that
    \begin{align} \label{eq:eps_descent_armijo}
        f \left( x + t \Bar{v} \right) \leq f(x) - t \| \Bar{v} \|^2 \quad \forall t \in (0,\eps / \| \Bar{v} \|],
    \end{align}
    so $\Bar{v}$ is a descent direction of $f$ at $x$. In particular, $t = \eps/\| \Bar{v} \|$ is an explicit step length that yields decrease in $f$. If instead $\Bar{v} = 0$, then no descent direction can be derived from $\partial_\eps f(x)$. However, since $\partial_{\eps'} f(x) \subseteq \partial_\eps f(x)$ for $\eps' < \eps$, a descent direction at $x$ may still be computable by reducing $\eps$.

    Unfortunately, the direction $\Bar{v}$ cannot be computed in practice, since it is based on knowing the entire $\eps$-subdifferential. To fix this issue, approximations $W$ of $\partial_\eps f(x)$ are considered and the direction
    \begin{align*}
        v := \argmin_{\xi \in - \conv(W)} \| \xi \|^2
    \end{align*}
    is computed. Clearly, we cannot use arbitrary subsets $W \subseteq \partial_\eps f(x)$ for this. Motivated by \eqref{eq:eps_descent_armijo}, we choose $c \in (0,1)$ and $\delta > 0$ and say that $W$ is a \emph{sufficient} approximation (and that $v$ yields \emph{sufficient} descent), if $\| v \| \leq \delta$ or
    \begin{align} \label{eq:sufficient_decrease}
        f \left( x + \frac{\eps}{\|v\|} v \right) \leq f(x) - c \eps \| v \|.
    \end{align}
    Based on these ideas, Algorithm \ref{algo:abstract_descent_method} can be constructed as an abstract descent method.
    \begin{algorithm}
    	\caption{Abstract $\varepsilon$-descent method}
    	\label{algo:abstract_descent_method}
    	\begin{algorithmic}[1] 
    		\Require Initial point $x^0 \in \R^n$, sequences $(\eps_j)_j, (\delta_j)_j \in \R^{>0}$, parameter $c \in (0,1)$. 
            \State Initialize $i = 0$, $j = 1$ and $x^{1,0} = x^0$.
    		\State Compute $v = \argmin_{\xi \in -\conv(W)} \| \xi \|^2$ for a sufficient (w.r.t. $c$) approximation $W \subseteq \partial_{\eps_j} f(x^{j,i})$.
            \If {$\| v \| \leq \delta_j$}
                \State Set $x^{j+1,0} = x^{j,i}$, $j = j+1$, $i = 0$ and go to Step 2.
            \Else
                \State Compute $t \geq \eps_j / \| v \|$ with $f(x^{j,i} + t v) \leq f(x^{j,i}) -c t \| v \|^2$.
                \State Set $x^{j,i+1} = x^{j,i} + t v$, $i = i+1$ and go to Step 2.
            \EndIf.
	       \end{algorithmic} 
    \end{algorithm}

    Note that there are two indices in Algorithm \ref{algo:abstract_descent_method}: The index $i$ enumerates the individual descent steps for fixed $\eps_j$ and $\delta_j$, and the index $j$ is increased whenever a point $x^{j,i}$ is reached where no further decrease can be achieved (with the tolerances $\eps_j$ and $\delta_j$). For $j \in \N$ let $N_j \in \N \cup \{ 0 \}$ be the final $i$ in Step 4 before $i$ is reset to $0$ and $j$ is increased. (In other words, $N_j$ is the number of descent steps that are executed for each $j$.) Since the iterates $x^{j,i}$ are generated sequentially, we can also enumerate them in a more classical way as
    \begin{align} \label{eq:def_z}
        (z^l)_l := (x^{1,0},x^{1,1},\dots,x^{1,N_1},x^{2,0},x^{2,1},\dots,x^{2,N_2},x^{3,0},\dots).
    \end{align}
    For ease of notation, let $(x^j)_j$ be the sequence with
    \begin{align} \label{eq:def_xj}
        x^j := x^{j,N_j} = x^{j+1,0} \quad \forall j \in \N.
    \end{align}
    Then by construction it holds $\min(\| \partial_{\eps_j} f(x^j) \|) \leq \delta_j$ for all $j \in \N$. If both $(\eps_j)_j$ and $(\delta_j)_j$ converge to 0, then by Lemma 4.4.4 in \cite{G2022}, all accumulation points of $(x^j)_j$ are critical points of $f$.

    Different implementable versions of Algorithm \ref{algo:abstract_descent_method}, with explicit ways to compute the direction $v$ in Step 2 and the step length $t$ in Step 6, can be found in the literature:
    \begin{itemize}
        \item In \cite{BLO2005,BCL2020}, the set $W$ is obtained by randomly sampling points $\{ y^1, \dots, y^{2n} \}$ from $\Bcl_{\eps_j}(x^{j,i}) \setminus \Omega$, where $\Omega$ is the set of points in which $f$ is not differentiable, and then setting $W = \{ \nabla f(y^1), \dots, \nabla f(y^{2n}) \}$. (Note that this may not yield a sufficient approximation as in \eqref{eq:sufficient_decrease}.) The step length is computed via an Armijo-like backtracking line search.
        \item In \cite{GP2021,G2024,G2022,MY2012}, the set $W$ is obtained deterministically by starting with an initial subset  $W_0 \subseteq \partial_{\eps_j} f(x^{j,i})$ (e.g., $W_0 = \{ \xi \}$ for $\xi \in \partial f(x^{j,i})$) and then iteratively adding new subgradients until the approximation is sufficient. For the step length, the Armijo-like line search from \cite{BLO2005,BCL2020} is used.
        \item In \cite{ZLJ2020}, the direction $v$ in Step 2 is obtained in an iterative fashion by starting with a subgradient in $\partial f(x^{j,i})$ as an initial direction, and then updating this direction by iteratively taking convex combinations with additional subgradients from $\partial_{\eps_j} f(x^{j,i})$. (More precisely, each additional subgradient is sampled randomly along the current direction $v$.) For the step length, $t = \eps_j / \| v \|$ is used. Furthermore, the sequences $(\eps_j)_j = \eps$ and $(\delta_j)_j = \delta$ are constant, so the method stops via Step 4 once an $(\eps,\delta)$-critical point is reached. (Note that in the notation of \cite{ZLJ2020}, the roles of $\eps$ and $\delta$ are reversed compared to our notation.)
    \end{itemize} 

    \subsection{Application of our results} \label{sec:descent_methods:application_of_results}
    
    In the following, we derive the speed of R-convergence of the sequence $(x^j)_j$ computed via Algorithm \ref{algo:abstract_descent_method}. By applying Theorem \ref{thm:xj_speed_higher_order} we immediately obtain the following corollary:
    \begin{corollary} \label{cor:GS_rate}
        Let $f : \R^n \rightarrow \R$ be locally Lipschitz continuous. Let $(\eps_j)_j \in \R^{>0}$ and $(\delta_j)_j \in \R^{>0}$ with $\eps_j \rightarrow 0$. Let $(x^j)_j$ be the sequence generated by Algorithm \ref{algo:abstract_descent_method} (cf. \eqref{eq:def_xj}). Assume that $x^j \rightarrow x^* \in \R^n$, where $x^*$ is a minimum of order $p \in \N$ with constant $\beta > 0$ for which \eqref{eq:p_order_min_semismooth} holds.
        \begin{enumerate}
            \item[(i)] If $p = 1$ and $\delta_j \rightarrow \Bar{\delta} < \beta$, then there are $M > 0$ and $N \in \N$ such that
            \begin{align} \label{eq:GS_rate_p_eq_1}
                \| x^j - x^* \| \leq M \eps_j \quad \forall j > N.
            \end{align}
            \item[(ii)] If $p \geq 2$ then there are $M > 0$ and $N \in \N$ such that
            \begin{align} \label{eq:GS_rate_p_lar_2}
                \| x^j - x^* \| \leq M \max(\eps_j^{\frac{1}{p}}, \delta_j^{\frac{1}{p-1}}) \quad \forall j > N.
            \end{align}
        \end{enumerate}
    \end{corollary}

    The previous corollary (combined with Remark \ref{remark:conv_exponent}) shows that we can technically use Algorithm \ref{algo:abstract_descent_method} to obtain sequences $(x^j)_j$ of arbitrary order of R-convergence by choosing sequences $(\eps_j)_j$ and $(\delta_j)_j$ with the respective order. However, clearly, the speed of convergence of $(x^j)_j$ alone is not a good measure for the efficiency of Algorithm \ref{algo:abstract_descent_method}: Since $N_j$ iterations are required to get from $x^{j-1} = x^{j,0}$ to $x^j = x^{j,N_j}$ and since there may be no upper bound for $(N_j)_j$, the effort for computing each $x^j$ may grow with $j$. As a thorough analysis of the boundedness of $(N_j)_j$ would go beyond the scope of this article, we leave this line of research for future work. A more useful quantity to measure the efficiency of Algorithm \ref{algo:abstract_descent_method} would be the speed of convergence of $(z^l)_l$. But since we do not have any useful upper bound for $\min(\| \partial_{\eps} f(z^l) \|)$ for the intermediate steps in-between $(x^j)_j$, the speed of $(z^l)_l$ cannot directly be inferred from Theorem \ref{thm:xj_speed_higher_order}. Nonetheless, there is a way to estimate the speed of the sequence of objective values $(f(z^l))_l$ using $(x^j)_j$, assuming that $(N_j)_j$ is bounded:
       
    \begin{lemma} \label{lem:fz_rate}
        Let $f : \R^n \rightarrow \R$ be locally Lipschitz continuous. Let $(x^j)_j$ be the sequence generated by Algorithm \ref{algo:abstract_descent_method} (cf. \eqref{eq:def_xj}).
        Assume that $x^j \rightarrow x^* \in \R^n$ and that $N_j \leq \Bar{N}$ for all $j \in \N$ for some $\Bar{N} \in \N$. If $r : \R \rightarrow \R^{\geq 0}$ is a monotonically decreasing function such that 
        \begin{align} \label{eq:fz_rate_r}
            \| x^j - x^* \| \leq r(j) \quad \forall j \in \N,
        \end{align}
        then
        \begin{align} \label{eq:fz_rate}
            f(z^l) - f(x^*) \leq L \Tilde{r}(l) \quad \text{for} \quad \Tilde{r}(l) := r \left( \frac{l}{\Bar{N} + 1} - 1 \right) \quad \forall l > N_1 + 1,
        \end{align}
        where $L$ is a Lipschitz constant of $f$ around $x^*$.
    \end{lemma}
    \begin{proof}
        For $l \in \N$ let $j_l \in \N$, $i_l \in \{ 0, \dots, N_{j_l} \}$ be the indices in the notation \eqref{eq:def_z} such that $x^{j_l,i_l} = z^l$. Then $l = j_l + i_l + \sum_{j = 1}^{j_l - 1} N_j$.
        Since $i_l \leq N_{j_l} \leq \Bar{N}$ it holds
        \begin{align*}
            l 
            \leq j_l + i_l + \Bar{N}(j_l - 1) 
            = (\Bar{N} + 1) j_l + i_l - \Bar{N}
            \leq (\Bar{N} + 1) j_l,
        \end{align*}
        i.e., $j_l \geq l / (\Bar{N} + 1)$.
        Let $L$ be a Lipschitz constant of $f$ on an open superset of $\{ z^l : l \in \N \} \cup \{ x^* \}$. Then
        \begin{align*}
            f(x^j) - f(x^*) \leq L \| x^j - x^* \| \leq L r(j) \quad \forall j \in \N.
        \end{align*}
        Since $(f(z^l))_l$ is a monotonically decreasing sequence by construction and $r$ is a monotonically decreasing function by assumption, we have
        \begin{align*}
            f(z^l) - f(x^*)
            &= f(x^{j_l,i_l}) - f(x^*)
            \leq f(x^{j_l,0}) - f(x^*)
            = f(x^{j_l - 1}) - f(x^*) \\
            &\leq L r(j_l - 1)
            \leq L r \left( \frac{l}{\Bar{N} + 1} - 1 \right) \quad \forall l > N_1 + 1,
        \end{align*}
        completing the proof. (Note that $l > N_1 + 1$ is required to have $j_l > 0$.)
    \end{proof}

    As an application of the previous lemma, assume that we are in case (i) of Corollary \ref{cor:GS_rate} with $\eps_j = \keps^j$ for some $\keps \in (0,1)$, i.e., $r(j) = M \keps^j$ and $(x^j)_j$ converges R-linearly with a rate of $\keps$. Then 
    \begin{align*}
        \Tilde{r}(l) 
        = r \left( \frac{l}{\Bar{N} + 1} - 1 \right)
        = M \keps^{\frac{l}{\Bar{N} + 1} - 1}
        = M \keps^{-1} (\keps^{\frac{1}{\Bar{N} + 1}})^l,
    \end{align*}
    so $f(z^l)_l$ still converges R-linearly with a rate of $\keps^{1/(\Bar{N}+1)}$. Unfortunately, higher orders of convergence are not preserved: If $r(j) = M \keps^{2^j}$ then $(r(j))_j$ converges Q-quadratically, but $(\Tilde{r}(l))_l$ only converges Q-superlinearly and not Q-quadratically (unless $\Bar{N} = 0$).  

    Example \ref{example:f_rate_sharp} in the appendix shows that the estimate in Lemma \ref{lem:fz_rate} is tight (up to constant factors). However, note that in the proof of Lemma \ref{lem:fz_rate}, we essentially calculated an estimate for the case where $N_j = \bar{N}$ for all $j \in \N$. So unless $N_j$ is close to constant, this may lead to a large overestimation. In particular, any overestimation that is already present in the estimate \eqref{eq:fz_rate_r} is amplified by this. (This can be seen in Example \ref{example:maxq} below.) As such, we believe that Lemma \ref{lem:fz_rate} has more theoretical than practical relevance.

    Finally, it is worth pointing out that for $p = 1$, $(\delta_j)_j$ does not have to vanish for Corollary \ref{cor:GS_rate} to be applicable. We will briefly discuss the implications of this in the outlook in Section \ref{sec:outlook}.

    \subsection{Numerical experiments}

    In the following, we analyze the behavior of Algorithm \ref{algo:abstract_descent_method} experimentally in light of Corollary \ref{cor:GS_rate}. For our computations, we have implemented the method described in \cite{GP2021,G2022}, with the bisection method from \cite{G2024}, in Matlab. In addition to the fully deterministic computation of $W$ in Step 2, we also added an option for initialization with a number of randomly sampled gradients from $\Bcl_\eps(x^{j,i})$, which makes the method behave similar to the classical gradient sampling method from \cite{BLO2005,BCL2020} (and, in particular, similar to the abstract method where $W = \partial_{\eps_j} f(x^{j,i})$). The code is available at \url{https://github.com/b-gebken/DGS}, including scripts for reproducing all results from this section. Concerning the different choices of parameters for the method we make in this section, we stress that they are mainly chosen in a way that nicely highlights the behavior of the method, without trying to achieve the best possible performance. Note that Corollary \ref{cor:GS_rate} only prescribes the speed of convergence of $(x^j)_j$ up to the unknown constant $M$, which has no impact on the rate or order of convergence. For our visualizations, we choose it a posteriori in a way that makes it easy to compare both sides of \eqref{eq:GS_rate_p_eq_1} and \eqref{eq:GS_rate_p_lar_2}.

    In the first experiment, we show that Corollary \ref{cor:GS_rate} yields a tight bound (up to $M$) for the speed of convergence of $(x^j)_j$. For the general Theorem \ref{thm:xj_speed_higher_order}, this was shown theoretically with the function from Example \ref{example:xj_speed_estim_equal}, and we can use the same function here:

    \begin{example} \label{example:Cor2_sharp}
        Consider the function $f$ from Example \ref{example:xj_speed_estim_equal} for $p = 3$, i.e.,
        \begin{align*}
            f : \R^2 \rightarrow \R, \quad x \mapsto \max(1/3 |x_1|^3, |x_2|).
        \end{align*}
        For the parameters of Algorithm \ref{algo:abstract_descent_method}, we choose
        \begin{align*}
            x^0 = (10,0)^\top, \ \eps_j = 0.01 \cdot 0.85^j, \ \delta_j = 20 \cdot 0.75^j, \ c = 0.9.
        \end{align*}
        The set $W$ in Step 2 is obtained by first evaluating the gradient in $100$ uniformly random points from $\Bcl_\eps(x^{j,i})$ and then deterministically adding gradients (as described in \cite{GP2021,G2022,G2024}) until the approximation is sufficient (cf. \eqref{eq:sufficient_decrease}). For the resulting sequences $(x^j)_j$ and $(z^l)_l$ (cf. \eqref{eq:def_xj} and \eqref{eq:def_z}), the distances to the minimum $x^* = (0,0)^\top$ are shown in Figure \ref{fig:example_num_sharp}.
        \begin{figure}
            \centering
            \parbox[b]{0.49\textwidth}{
                \centering 
                \includegraphics[width=0.45\textwidth]{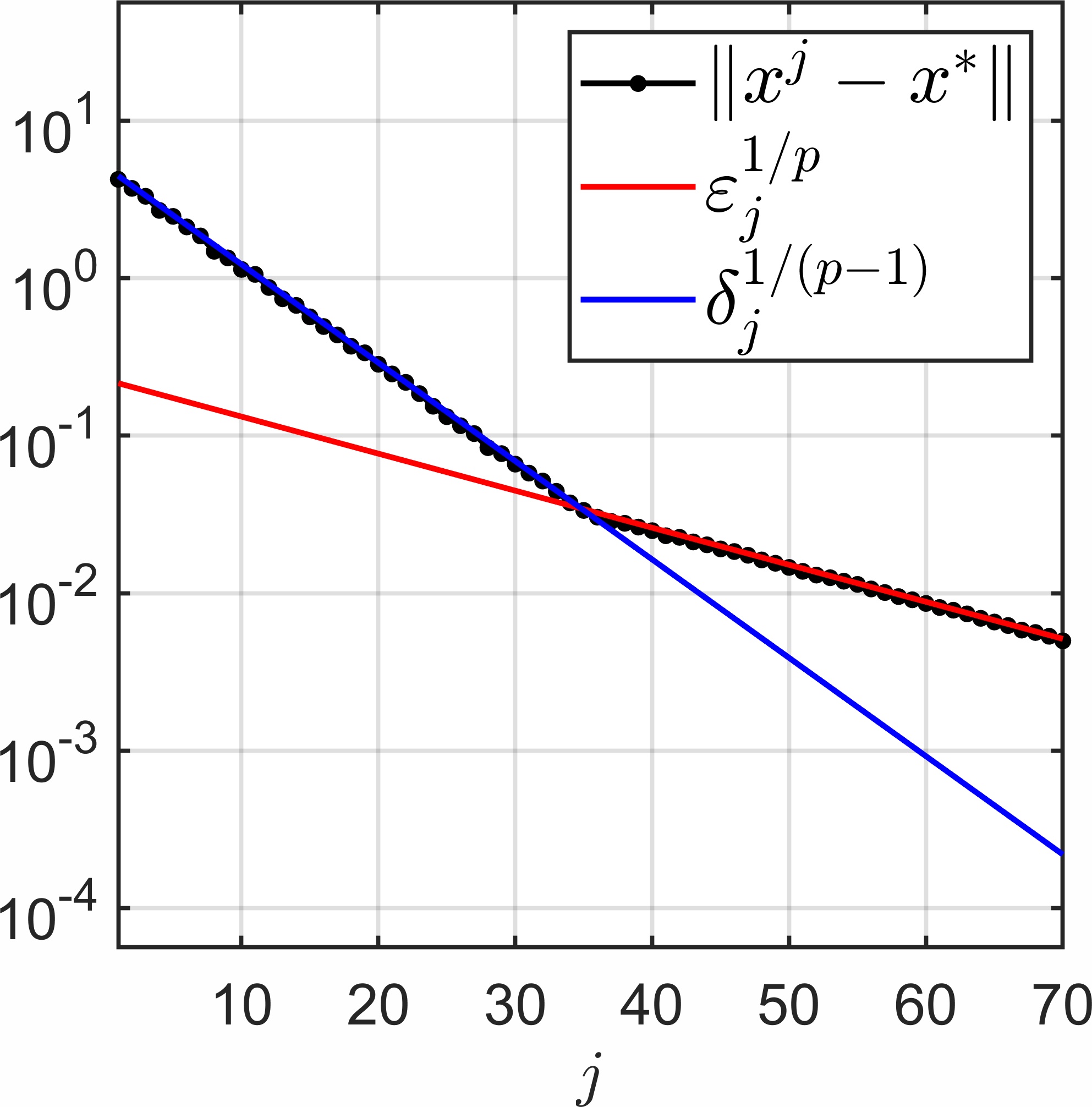}\\
                (a)
    		}
            \parbox[b]{0.49\textwidth}{
                \centering 
                \includegraphics[width=0.45\textwidth]{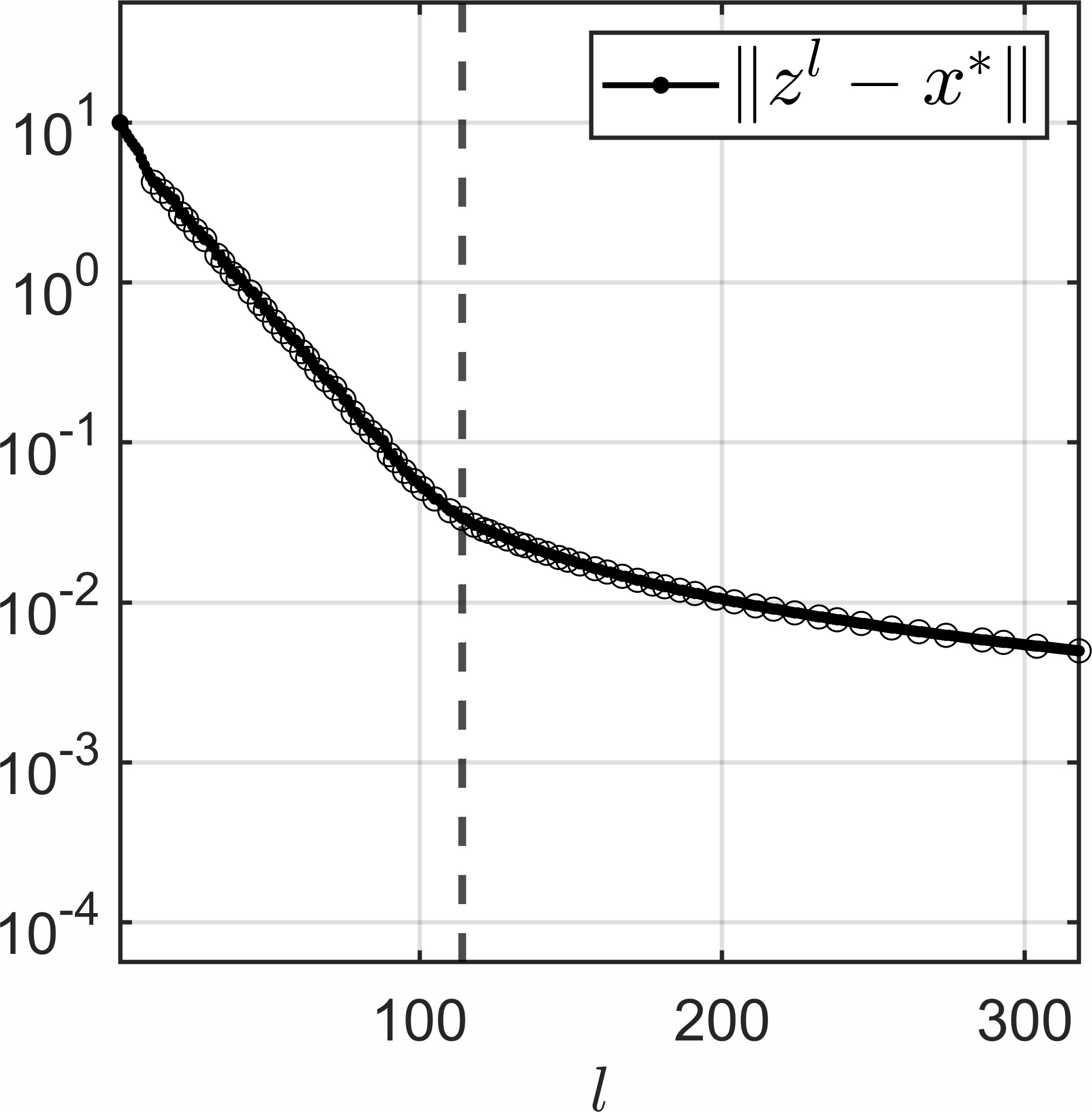}\\
                (b)
    		}
            \caption{(a) The sequences $(\| x^j - x^* \|)_j$, $(\eps_j)_j$ and $(\delta_j)_j$ in Example \ref{example:Cor2_sharp}. (b) The full sequence $(\| z^l - x^* \|)_l$. The circle markers indicate the indices at which $(x^j)_j$ appears within $(z^l)_l$, i.e., at which $\eps_j$ and $\delta_j$ change. The dashed line shows the the iteration at which $\delta_j$ becomes smaller than $\eps_j$ (cf. (a)).}
            \label{fig:example_num_sharp}
        \end{figure}%
        In Figure \ref{fig:example_num_sharp}(a) we see that $\| x^j - x^* \|$ is close to the maximum of $\eps_j^{1/p}$ and $\delta_j^{1/(p-1)}$ (i.e., $M \approx 1$ in Corollary \ref{cor:GS_rate}). For $j < 35$, the maximum equals $\eps_j^{1/p}$ and for $j \geq 35$, it equals $\delta_j^{1/(p-1)}$. In particular, the rate of convergence abruptly changes at $j = 35$. Figure \ref{fig:example_num_sharp}(b) shows the entire sequence $(z^l)_l$ produced by the algorithm. The circle markers highlight the iterates at which $(x^j)_j$ appears within $(z^l)_l$ and the dashed line shows the iterate $l$ where $x^{35} = z^l$. Also here, although less abrupt, we see a change in the rate of decrease of $\| z^l - x^* \|$.
    \end{example}
    
    In the next example, we consider the well-known test function \emph{MAXQ} from \cite{S1989,HMM2004}. This function was also considered in \cite{HSS2017}, where it was an example for a function to which the results from said article cannot be applied. (Actually, it was shown that the choice of $(\eps_j)_j$ and $(\delta_j)_j$ suggested in \cite{HSS2017}, Section 5, leads to slow convergence of the gradient sampling method from \cite{BLO2005}.) In contrast to this, Corollary \ref{cor:GS_rate} is applicable:

    \begin{example} \label{example:maxq}
        Consider the function
        \begin{align*}
            f : \R^n \rightarrow \R, \quad x \mapsto \max_{i \in \{1,\dots,n\}} x_i^2.
        \end{align*}
        It is easy to show that $x^* = (0,\dots,0)^\top \in \R^n$ is the unique global minimum of $f$ with order $p = 2$ and constant $1/n$. Furthermore, by Lemma \ref{lem:convex_higher_order}, property \eqref{eq:p_order_min_semismooth} holds in $x^*$. Let $n = 10$. For the parameters of Algorithm \ref{algo:abstract_descent_method}, we choose
        \begin{align*}
            x^0 = (1,\dots,5,-6,\dots,-10)^\top, \
            \eps_j = 10 \cdot 0.5^j, \
            \delta_j = 10 \cdot 0.5^j, \
            c = 0.9.
        \end{align*}
        The set $W$ is determined as in Example \ref{example:Cor2_sharp}. Figure \ref{fig:example_maxq}(a) shows the distance of the resulting $(x^j)_j$ to the minimum.
        \begin{figure}
            \centering
            \parbox[b]{0.32\textwidth}{
                \centering 
                \includegraphics[width=0.32\textwidth]{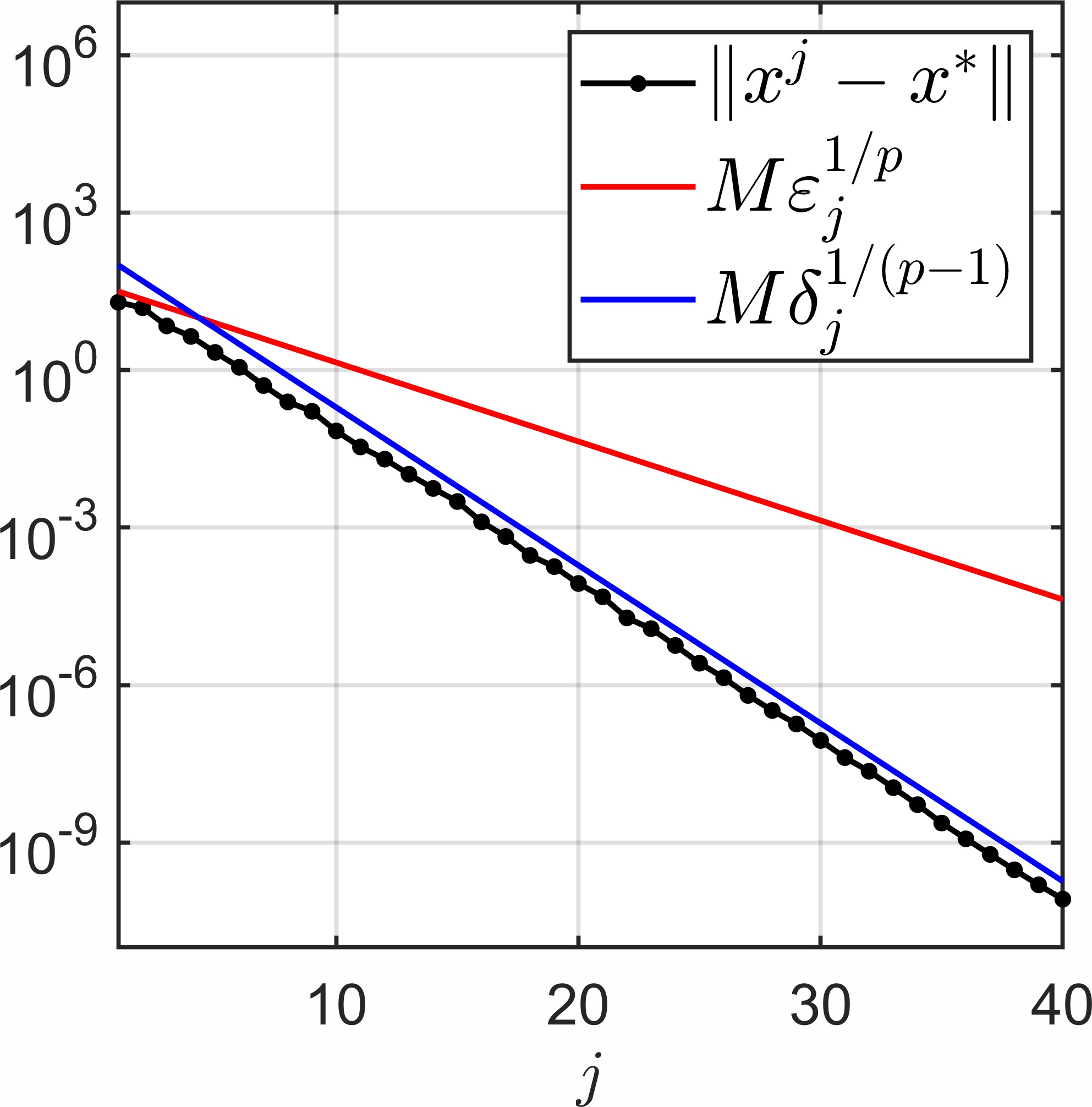}\\
                (a)
    		}
            \parbox[b]{0.32\textwidth}{
                \centering 
                \includegraphics[width=0.32\textwidth]{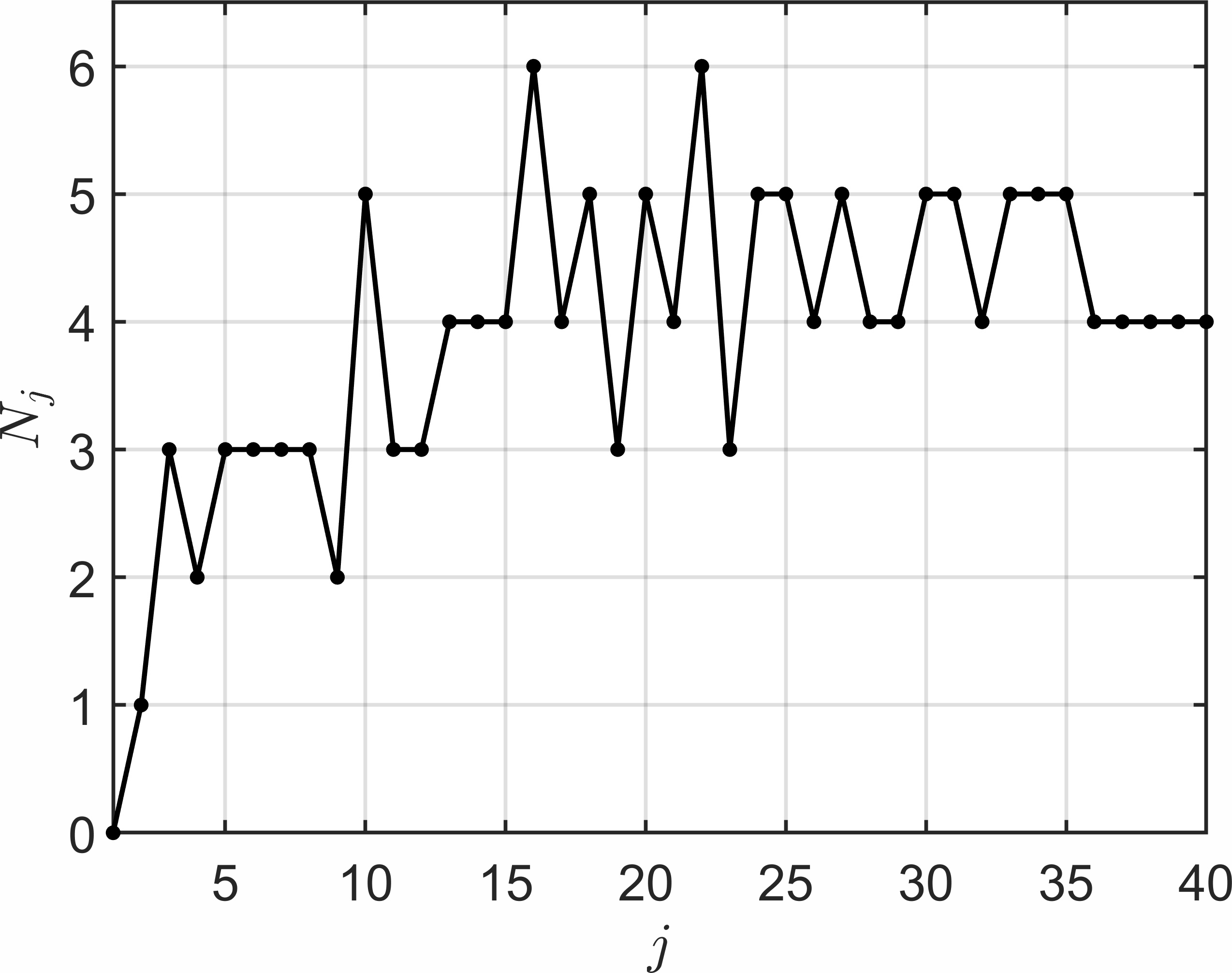}\\
                (b)
    		}
            \parbox[b]{0.32\textwidth}{
                \centering 
                \includegraphics[width=0.32\textwidth]{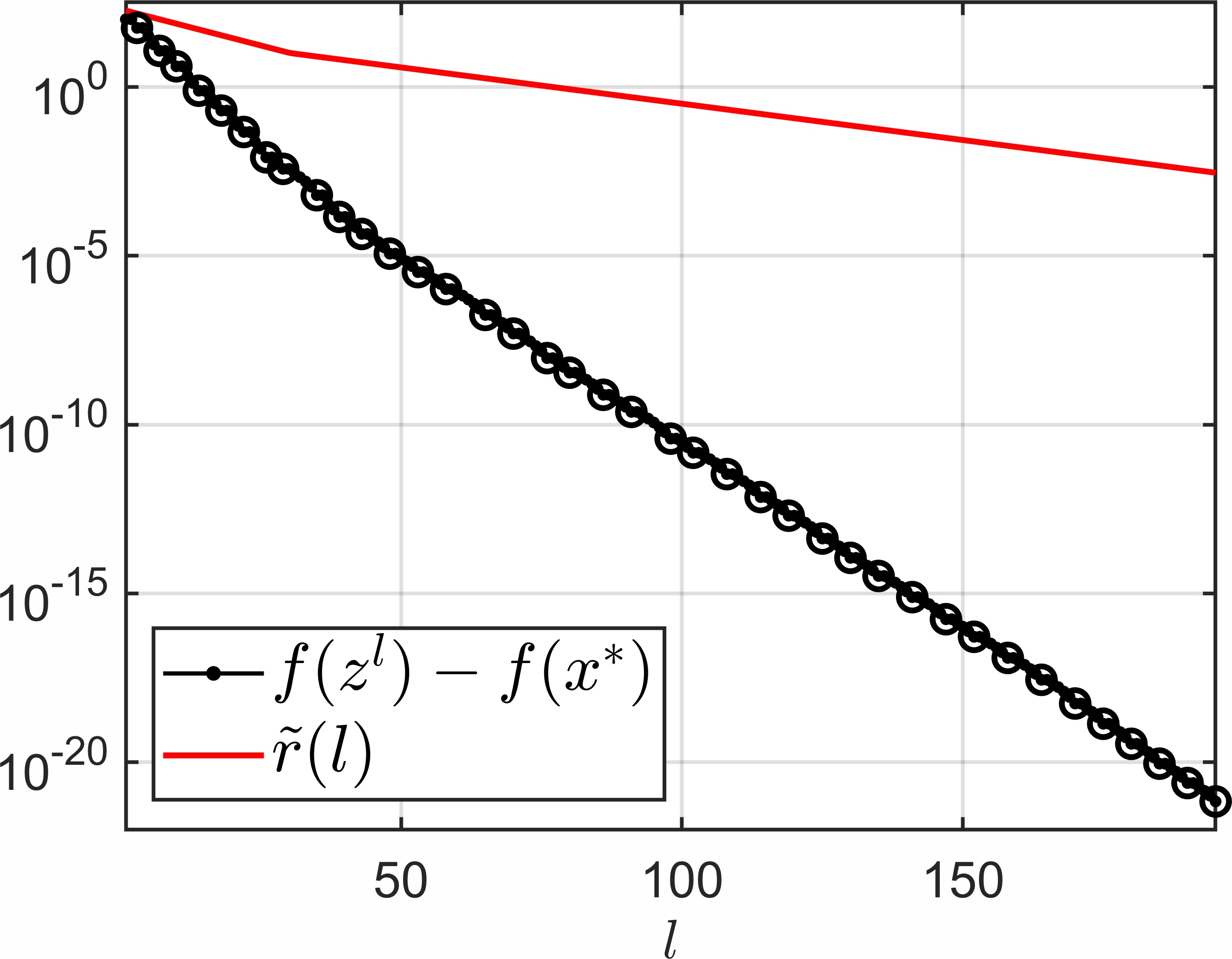}\\
                (c)
    		}
            \caption{(a) The sequences $(\| x^j - x^* \|)_j$, $(M \eps_j^{1/p})_j$ and $(M \delta_j^{1/(p-1)})_j$ in Example \ref{example:maxq} for $M = 10$. (b) The number of iterations for each $j$. (c) The sequences $(f(z^l) - f(x^*))_l$ and $(\Tilde{r}(l))_l$ from Lemma \ref{lem:fz_rate} with $r$ chosen as the right-hand side of \eqref{eq:GS_rate_p_lar_2}.}
            \label{fig:example_maxq}
        \end{figure}%
        As expected, it suggests that $(x^j)_j$ converges R-linearly. But in contrast to Example \ref{example:Cor2_sharp}, we do not obtain a tight upper bound from Corollary \ref{cor:GS_rate}. Figure \ref{fig:example_maxq}(b) shows the sequence $(N_j)_j$, i.e., the number of iterations Algorithm \ref{algo:abstract_descent_method} executed for each fixed $j$. It suggests that $(N_j)_j$ is bounded, so by Lemma \ref{lem:fz_rate}, $(f(z^l))_l$ converges R-linearly as well. Figure \ref{fig:example_maxq}(c) shows the distance of $(f(z^l))_l$ to $f(x^*)$ and indeed, it appears to converge R-linearly. However, as discussed at the end of Section \ref{sec:descent_methods:application_of_results}, we also see that Lemma \ref{lem:fz_rate} only yields a rough overestimate for the actual speed of convergence.
    \end{example}

    In the previous two examples, we constructed $(\eps_j)_j$ and $(\delta_j)_j$ so that they vanish Q-linearly. This is reasonable, since due to the first-order nature of Algorithm \ref{algo:abstract_descent_method}, we can only expect it to generate sequences with R-linear convergence at best. In our final example, we demonstrate the effect of choosing Q-superlinearly vanishing sequences instead: 

    \begin{example} \label{example:superlinear}
        Consider the function
        \begin{align*}
            f : \R^n \rightarrow \R, \quad x \mapsto \max_{i \in \{1,\dots,n\}} x_i + \frac{1}{2} \| x \|^2,
        \end{align*}
        which belongs to the class of functions introduced in \cite{N2004}, Section 3.2.1. It is easy to show that $x^* = (-1/n,\dots,-1/n)^\top \in \R^n$ is the unique global minimum of $f$ with order $p = 2$ and constant $\beta = 1/2$. Furthermore, by Lemma \ref{lem:convex_higher_order}, property \eqref{eq:p_order_min_semismooth} holds in $x^*$. Let $n = 100$. For the parameters of Algorithm \ref{algo:abstract_descent_method}, we choose
        \begin{align*}
            x^0 = (10,\dots,10)^\top, \
            \eps_j = 10 \cdot 0.75^{j^2}, \
            \delta_j = 10 \cdot 0.75^{j^2}, \
            c = 0.9.
        \end{align*}
        The set $W$ in Step 2 is obtained as in \cite{GP2021,G2022,G2024} (without any random sampling). The distance of the resulting sequence $(x^j)_j$ to the minimum is shown in Figure \ref{fig:example_superlinear}(a).
        \begin{figure}
            \centering
            \parbox[b]{0.49\textwidth}{
                \centering 
                \includegraphics[width=0.35\textwidth]{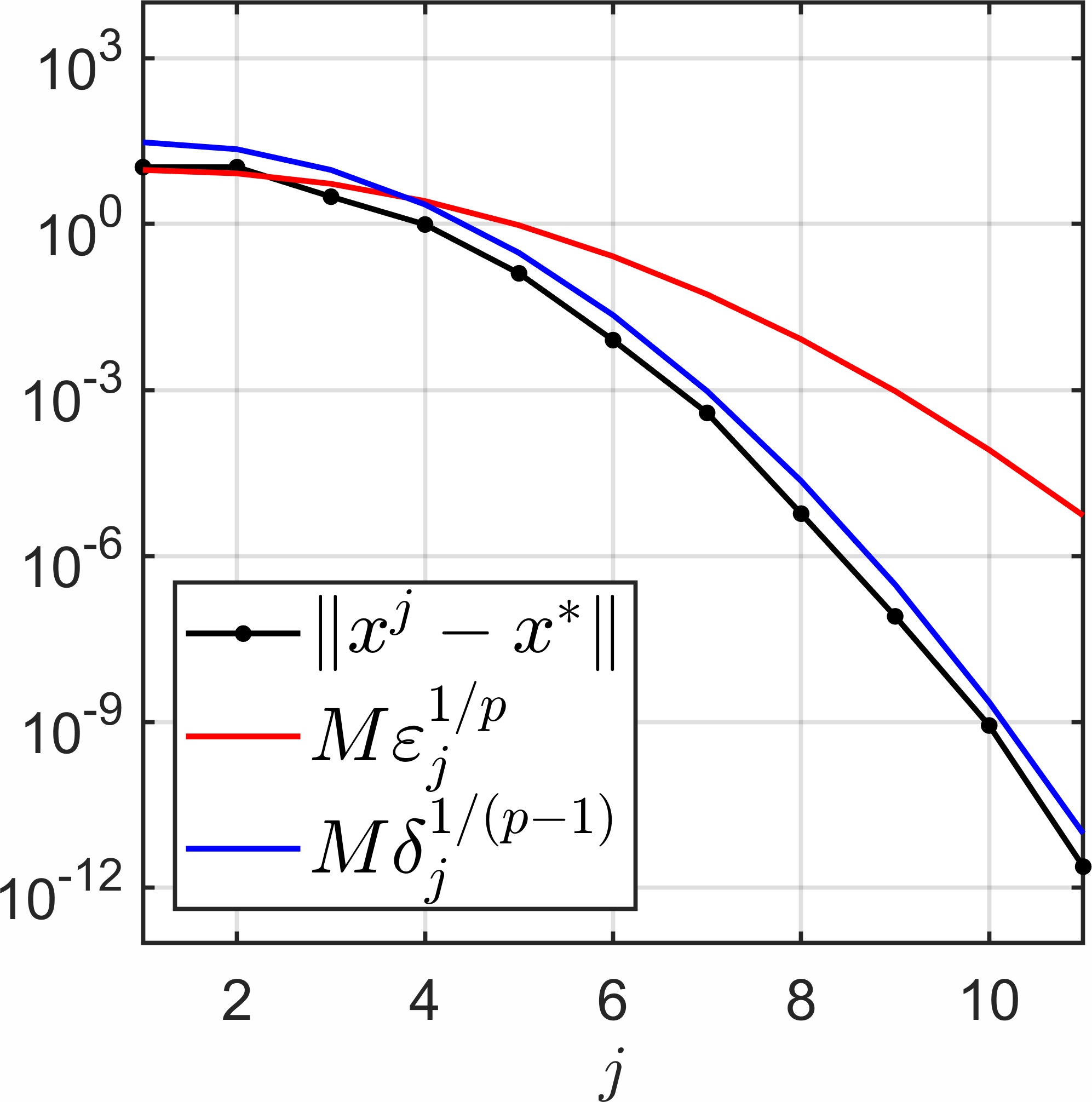}\\
                (a)
    		}
            \parbox[b]{0.49\textwidth}{
                \centering 
                \includegraphics[width=0.46\textwidth]{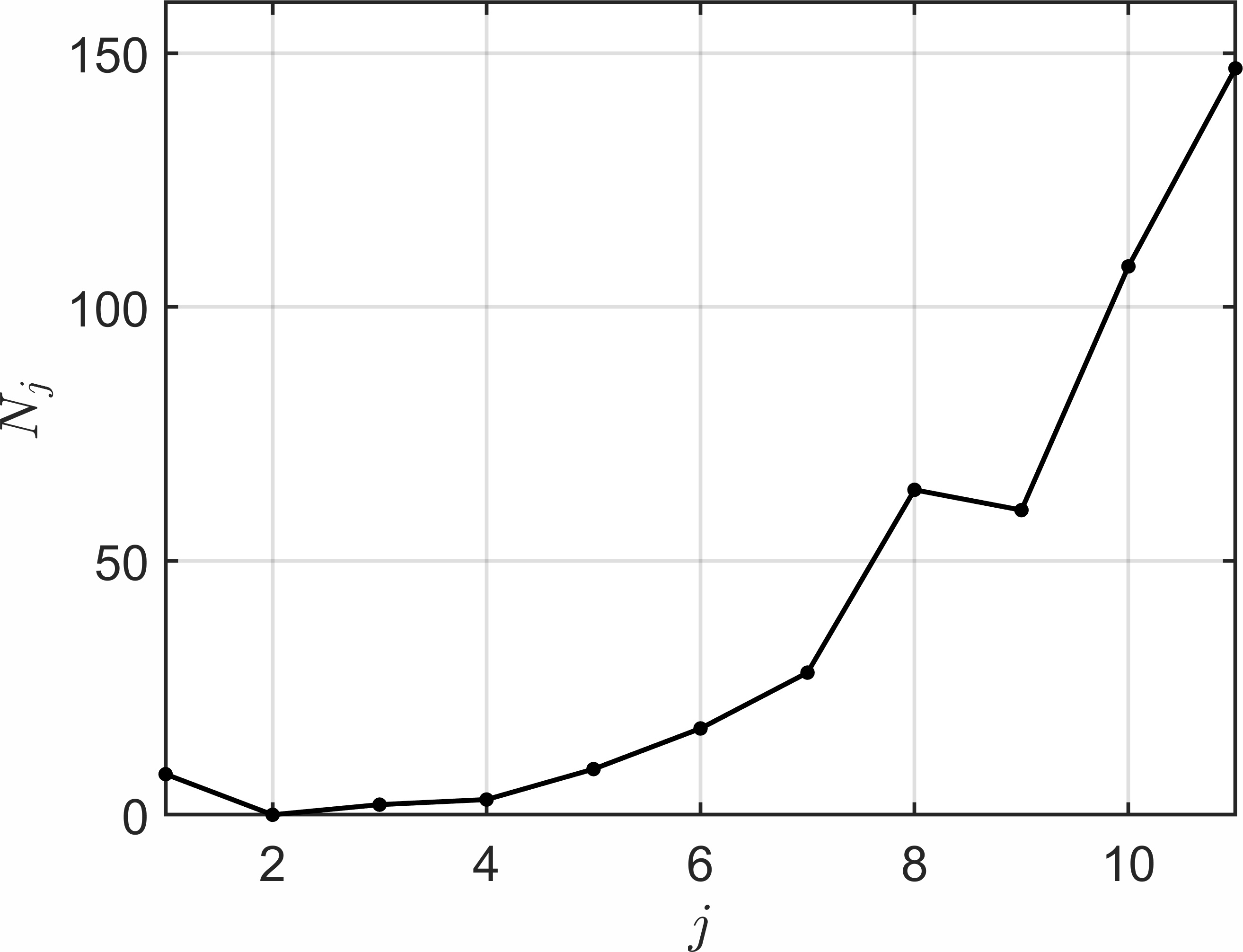}\\
                (b)
    		}
            \caption{(a) The sequences $(\| x^j - x^* \|)_j$, $(M \eps_j^{1/p})_j$ and $(M \delta_j^{1/(p-1)})_j$ in Example \ref{example:superlinear} for $M = 3$. (b) The number of iterations $N_j$ for each $j$.}
            \label{fig:example_superlinear}
        \end{figure}%
        As expected from Corollary \ref{cor:GS_rate}, $(x^j)_j$ appears to converge R-superlinearly. However, Figure \ref{fig:example_superlinear}(b) shows that the number of iterations required for producing each $x^j$ grows exponentially. Thus, at least in this case, there appears to be no benefit when choosing Q-superlinearly vanishing $(\eps_j)_j$ and $(\delta_j)_j$.
    \end{example}

    \section{Conclusion and outlook} \label{sec:outlook}

    In this article, we showed that for sequences $(x^j)_j \in \R^n$, $(\eps_j)_j, (\delta_j)_j \in \R^{\geq 0}$ with $x^j \rightarrow x^*$ and $\min(\| \partial_{\eps_j} f(x^j) \|) \leq \delta_j$ for all $j \in \N$, the speed of convergence of $(x^j)_j$ can be derived from the speeds of $(\eps_j)_j$ and $(\delta_j)_j$ (Theorem \ref{thm:xj_speed_higher_order}), provided that $x^*$ satisfies a polynomial growth property (Definition \ref{def:order_minimum}) and the higher-order semismoothness property \eqref{eq:higher_order_semismooth}. If $f$ grows linearly around $x^*$, then \eqref{eq:higher_order_semismooth} is implied by standard semismoothness. If the order of growth is higher than linear, then \eqref{eq:higher_order_semismooth} is more difficult to verify. However, we were able to show that piecewise differentiability and a convexity assumption w.r.t. the higher-order derivatives \eqref{eq:higher_order_pos_def} are sufficient for \eqref{eq:higher_order_semismooth}. As an application, we considered descent methods based on the Goldstein $\eps$-subdifferential, where $(\eps_j)_j$ and $(\delta_j)_j$ are inputs and $(x^j)_j$ is the output of an algorithm. In numerical experiments, we showed how our results can be used to predict and control the behavior of the algorithm.

    Fur future work, there are multiple interesting directions:
    \begin{itemize}
        \item For minima of order $p = 2$, we required the convexity assumption \eqref{eq:higher_order_pos_def} (see also Corollary \ref{cor:p_eq_2}). In \cite{HSS2017}, it appears that convexity was not needed, so there might be a way for us to drop this assumption as well. However, since our higher-order semismoothness property \eqref{eq:higher_order_semismooth} and also property \eqref{eq:p_order_min_semismooth} may fail to hold in the nonconvex case (cf. Example \ref{example:crescent}), we would have to find a new way to solve the issue highlighted in Example \ref{example:why_higher_order_semismooth}.
        \item For minima of order $p = 1$, $(\delta_j)_j$ is not required to vanish for Corollary \ref{cor:GS_rate} to be applicable. This could have interesting implications for gradient sampling methods, since $\| v \| \geq \bar{\delta}$ means that the theoretical descent we can estimate from \eqref{eq:sufficient_decrease} would improve (compared to vanishing $\| v \|$). The general convergence of the method would have to be proven again, but we expect that this is possible. 
        \item Combination of Corollary \ref{cor:GS_rate} and Lemma \ref{lem:fz_rate} yields a connection between the speed of convergence of Algorithm \ref{algo:abstract_descent_method} and boundedness of the sequence $(N_j)_j$. We believe that this may enable new proofs of linear convergence of Algorithm \ref{algo:abstract_descent_method} for classes of objective functions which could not be treated in \cite{HSS2017}. In particular, this could provide new guidelines how to choose the parameters $(\eps_j)_j$ and $(\delta_j)_j$ to obtain good performance.
        \item Throughout Section \ref{sec:descent_methods} we carefully only spoke about the speed of convergence of the sequence(s) generated by Algorithm \ref{algo:abstract_descent_method}. To properly analyze the efficiency of the algorithm itself, one has to also factor in the cost of computing the approximation $W$ in Step 2. For classical gradient sampling, this cost is clearly fixed, with the downside that insufficient approximations may be generated. For the method in \cite{GP2021,G2022,G2024}, to the best of the authors' knowledge, no upper bound on the cost for computing $W$ (maybe depending on the dimension $n$ and a Lipschitz constant $L$) has been proven so far.
        \item Since Corollary \ref{cor:GS_rate} is not limited to linear convergence, it might be usable as a tool to analyze the speed of convergence of higher-order methods for nonsmooth optimization like \cite{LV1998,MS2005,LO2013,G2022b}, or to inspire entirely new higher-order methods. 
    \end{itemize}

\backmatter

\noindent \textbf{Acknowledgements.} \quad This research was funded by Deutsche Forschungsgemeinschaft (DFG, German Research Foundation) – Projektnummer 545166481.

\begin{appendices}

\section{}

\begin{example} \label{example:higher_order_convex}
        Consider the function $f : \R^2 \rightarrow \R$, $x \mapsto \max_{i \in \{1,\dots,4\}} f_i(x)$ with
        \begin{alignat*}{2}
            &f_1(x) = (x_1 + 1)^3 + x_2^3, && \quad f_2(x) = (x_1 + 1)^3 - x_2^3,  \\
            &f_3(x) = -(x_1-1)^3 + x_2^3,  && \quad f_4(x) = -(x_1-1)^3 - x_2^3.
        \end{alignat*}
        The graph of $f$ is shown in Figure \ref{fig:example_higher_order_convex_and_crescent}(a).
        \begin{figure}
            \centering
            \parbox[b]{0.49\textwidth}{
                \centering 
                \includegraphics[width=0.49\textwidth]{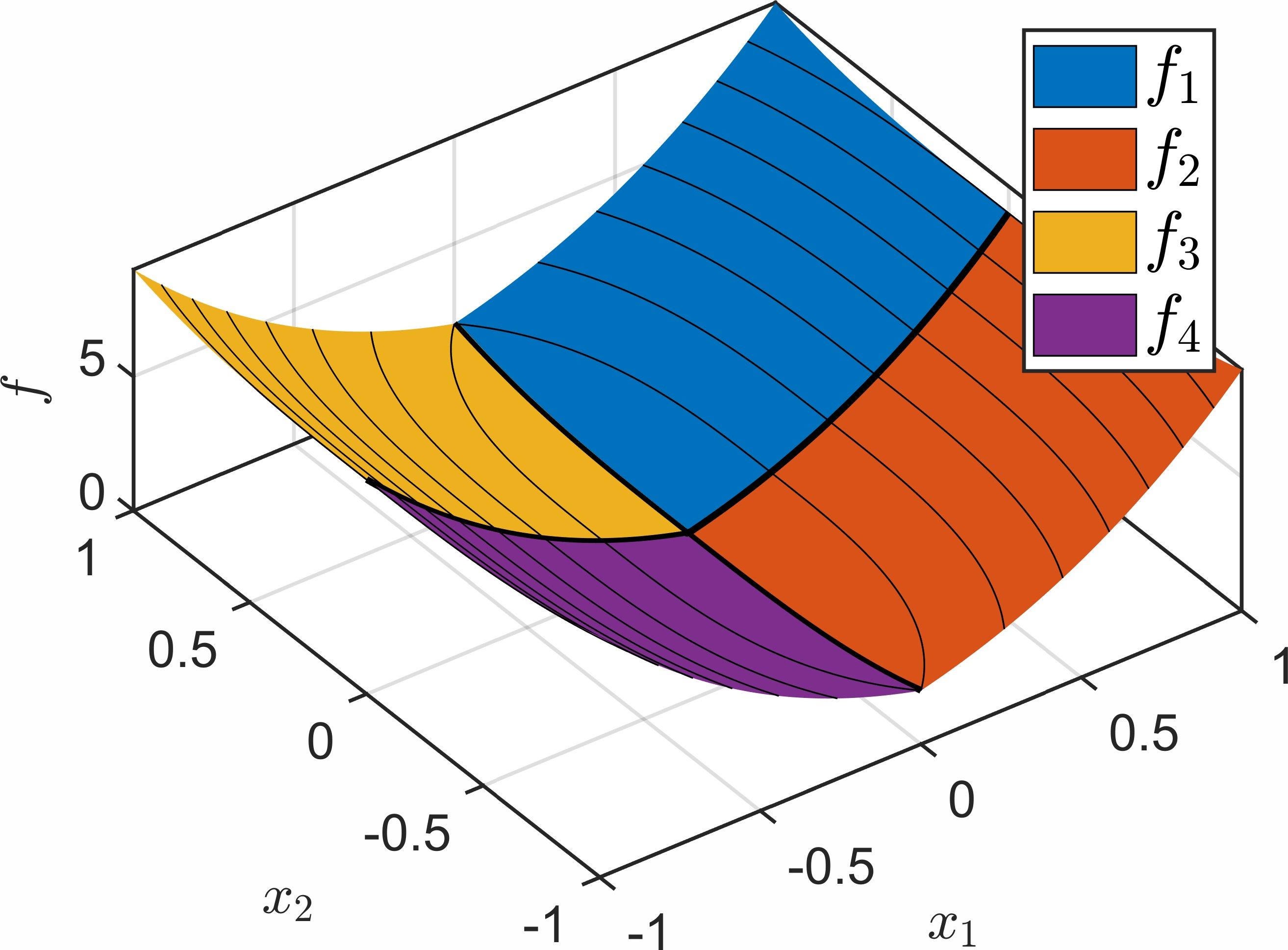}\\
                (a)
    		}
            \parbox[b]{0.49\textwidth}{
                \centering 
                \includegraphics[width=0.49\textwidth]{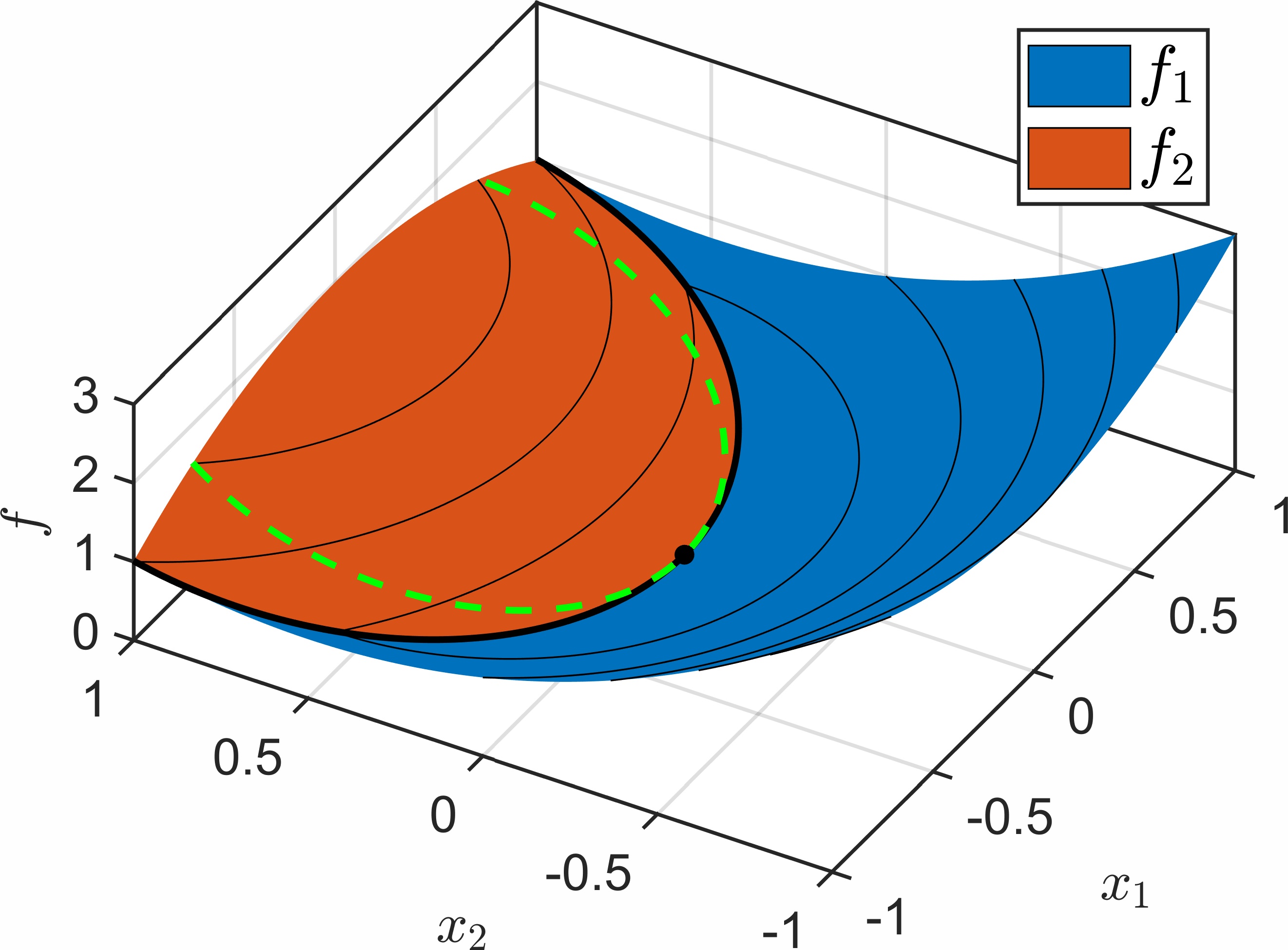}\\
                (b)
    		}
            \caption{(a) The graph $f$ in Example \ref{example:higher_order_convex}. (b) The graph of $f$ (Crescent) in Example \ref{example:crescent}. The dashed green line shows the set $D_{3/4}$ (mapped onto the graph).}
            \label{fig:example_higher_order_convex_and_crescent}
        \end{figure}  
        It is easy to see that $x^* = (0,0)^\top$ is the unique global minimum with $f(x) \geq f(x^*) + \| x \|^3$ for all $x \in \R^2$ (and $f(x) = f(x^*) + \| x \|^3$ for $x \in \{ 0 \} \times \R$), so $x^*$ is a minimum of order $3$ with constant $\beta = 1$. Computing the higher-order derivatives of $f_1$ in $x^*$, we obtain
        \begin{align*}
            \deriv^{(2)} f_1(x^*)(d)^2 &= 6 d_1^2, \\
            \deriv^{(3)} f_1(x^*)(d)^3 &= 6 (d_1^3 + d_2^3) 
        \end{align*}
        for all $d \in \R^2$. Thus for the open set $V_1 := \{ d \in \R^n : d_1 > -d_2 \}$ it holds $\deriv^{(3)} f_1(x^*)(d)^3 \geq 0$ for all $d$. Since $f_1$ is active (i.e., $1 \in A(x)$) if and only if $x_1 \geq 0$ and $x_2 \geq 0$, we have $C_1(x^*) = \R^{\geq 0} \times \R^{\geq 0} \subseteq V_1$. Analogously, it can be shown that \eqref{eq:higher_order_pos_def} holds for $f_2$, $f_3$ and $f_4$ as well, such that Lemma \ref{lem:convex_higher_order} can be applied to see that $f$ has the higher-order semismoothness property \eqref{eq:higher_order_semismooth}.
    \end{example}

    \begin{example} \label{example:crescent}
        Consider the function $f : \R^2 \rightarrow \R$, $x \mapsto \max(f_1(x),f_2(x))$ with
        \begin{align*}
            &f_1(x) = x_1^2 + (x_2 - 1)^2 + x_2 - 1, \\
            &f_2(x) = -x_1^2 - (x_2 - 1)^2 + x_2 + 1,
        \end{align*}
        which is the well-known test function \emph{Crescent} from \cite{K1985} with unique global minimum $x^* = (0,0)^\top$. Its graph is shown in Figure \ref{fig:example_higher_order_convex_and_crescent}(b). The set of nonsmooth points $\Omega$ of $f$ is a circle with radius $1$ around $(0,1)^\top$. It is easy to show that $x^*$ is a minimum of order $2$ (with constant $\beta = 1/2$). 
        Now for some $r \in (0,1)$ let $d' \in \R^n$ with $\| d'\| = 1$ and $t > 0$ so that
        \begin{align*}
            x^* + t d' \in \{ (r \cos(\theta), r \sin(\theta) + r)^\top : \theta \in [0, 2 \pi) \} =: D_r.
        \end{align*}
        (For example, $D_{3/4}$ is shown in Figure \ref{fig:example_higher_order_convex_and_crescent}(b).) Then $A(x^* + t d') = \{ 2 \}$ and a straight-forward calculation shows that
        \begin{align*}
            \frac{\langle \nabla f_2(x^* + t d'), d' \rangle}{t}
            = \frac{-4 r + 3}{2 r},
        \end{align*}
        where the right-hand side does not depend on $t$ and $d'$. Thus, for $r = 3/4$, the fraction on the left-hand side of \eqref{eq:p_order_min_semismooth} is zero and for $r \in (3/4,1)$, it is even negative. 
    \end{example}

    \begin{example} \label{example:xj_speed_estim_equal}
        For $p \in \N$ consider the function
        \begin{align*}
            f : \R^2 \rightarrow \R, \quad x \mapsto \max(p^{-1} |x_1|^p, |x_2|) = \max(p^{-1} x_1^p, -p^{-1} x_1^p, x_2, -x_2).
        \end{align*}
        It is easy to see that $x^* = (0,0)^\top$ is the unique minimum of $f$ with $f(x^*) = 0$. The graph of $f$ and set of nonsmooth points
        \begin{align*}
            \Omega = \{ (t,p^{-1} |t|^p) : r \in \R \} \cup \{ (t,-p^{-1} |t|^p) : r \in \R \}
        \end{align*}
        are shown in Figure \ref{fig:example_xj_speed_estim_equal_and_f_rate_sharp}(a) for $p = 2$.
        \begin{figure}
            \centering
            \parbox[b]{0.32\textwidth}{
                \centering 
                \includegraphics[width=0.32\textwidth]{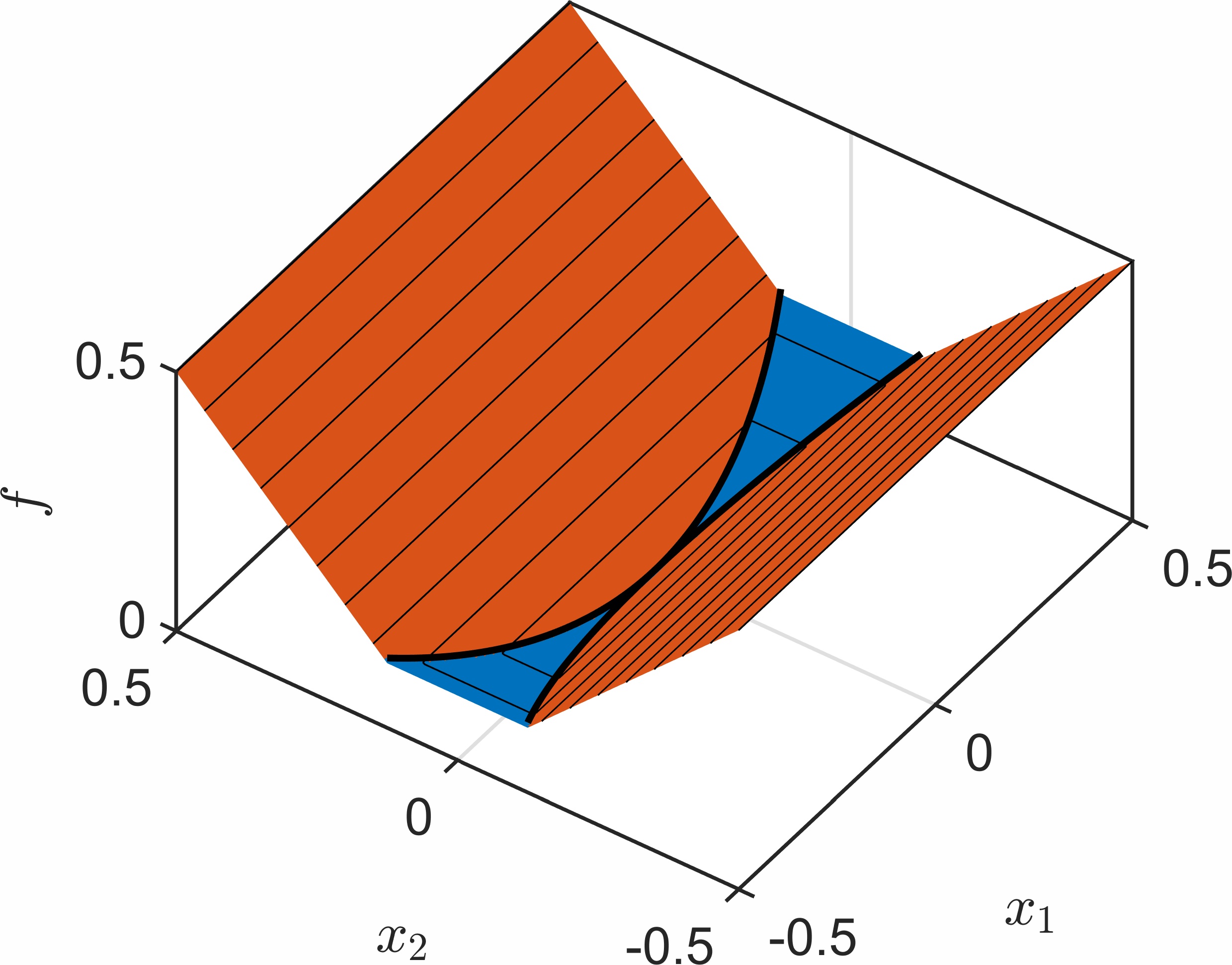}\\
                (a)
    		}
            \parbox[b]{0.32\textwidth}{
                \centering 
                \includegraphics[width=0.32\textwidth]{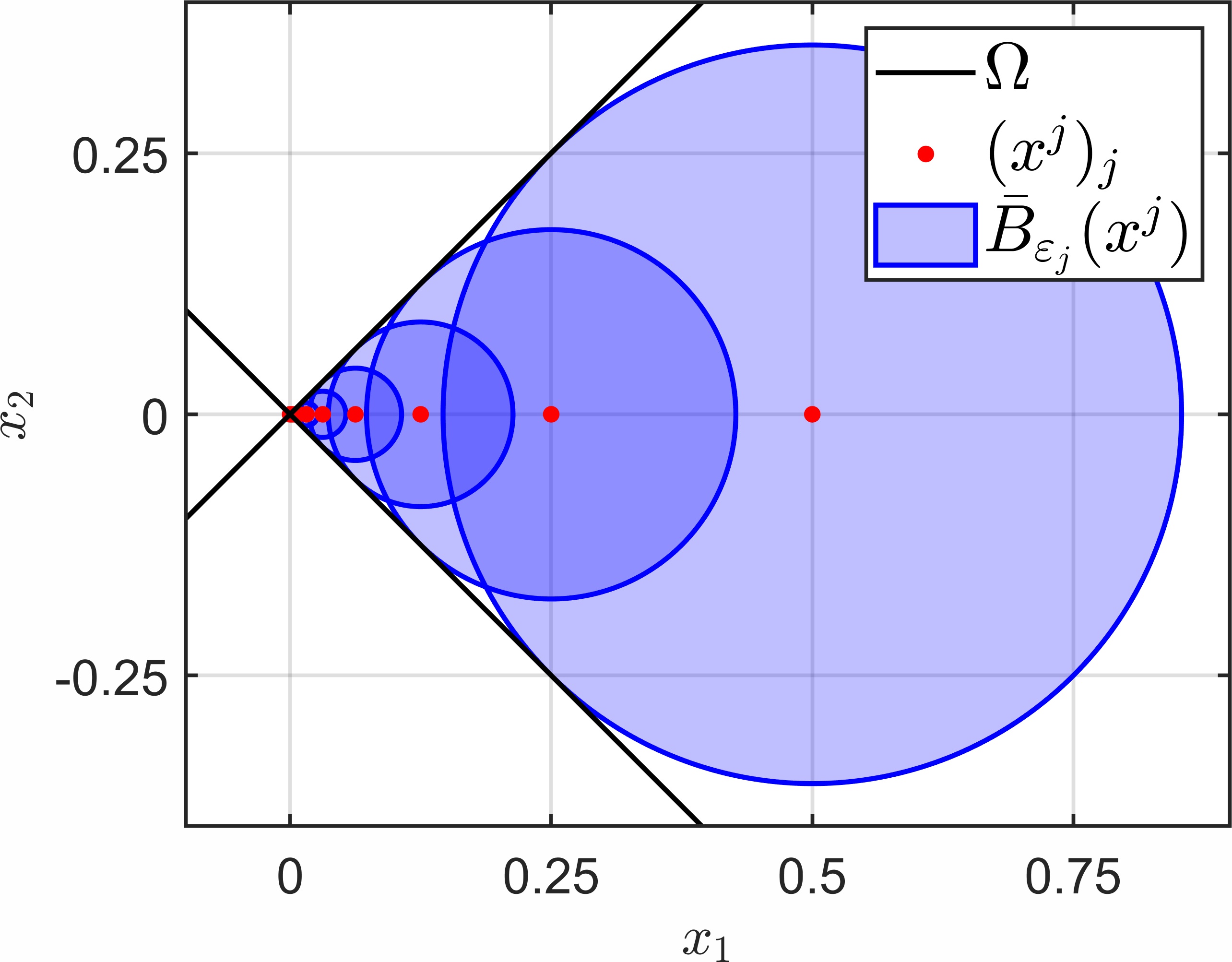}\\
                (b)
    		}
            \parbox[b]{0.32\textwidth}{
                \centering 
                \includegraphics[width=0.31\textwidth]{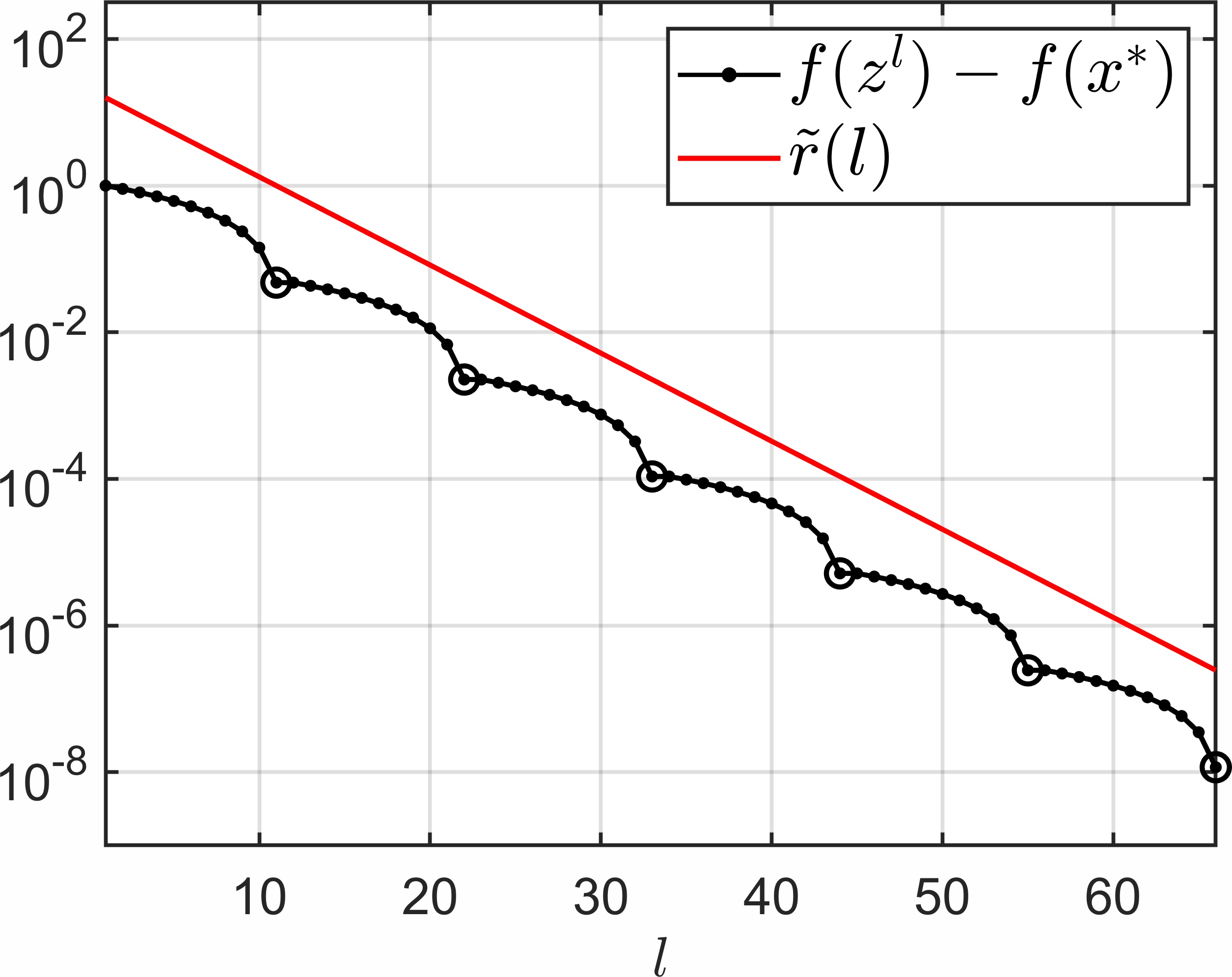}\\
                (c)
    		}
            \caption{(a) The graph of $f$ in Example \ref{example:xj_speed_estim_equal} for $p = 2$. (b) The set of nonsmooth points $\Omega$, the sequence $(x^j)_j$ and the ball $\Bcl_{\eps_j}(x^j)$ in Example \ref{example:xj_speed_estim_equal}(i) for $\kappa = 1/2$. (c)  The sequences $(f(z^l) - f(x^*))_l$ and $(\Tilde{r}(l))_l$ in Example \ref{example:f_rate_sharp}. The black circles show the subsequence $(f(x^j) - f(x^*))_j$ within $(f(z^l) - f(x^*))_l$ (cf. \eqref{eq:def_z}, \eqref{eq:def_xj}).}
            \label{fig:example_xj_speed_estim_equal_and_f_rate_sharp}
        \end{figure}   
        If $x \in \R^2$ with $p^{-1} |x_1|^p \geq |x_2|$, then 
        \begin{align*}
            \| x - x^* \|^p
            &= (x_1^2 + x_2^2)^{\frac{p}{2}}
            \leq (x_1^2 + p^{-2} x_1^{2p})^{\frac{p}{2}}
            = (p^{-\frac{2}{p}} x_1^2 (p^{\frac{2}{p}} + p^{-2 + \frac{2}{p}} x_1^{2(p-1)}))^{\frac{p}{2}} \\
            &= p^{-1} |x_1|^{p} (p^{\frac{2}{p}} + p^{-2 + \frac{2}{p}} x_1^{2(p-1)})^{\frac{p}{2}}
            = f(x) (p^{\frac{2}{p}} + p^{-2 + \frac{2}{p}} x_1^{2(p-1)}))^{\frac{p}{2}}.
        \end{align*}
        If, on the other hand, $p^{-1} |x_1|^p \leq |x_2|$, then
        \begin{align*}
            \| x - x^* \|^{p}
            &= (x_1^2 + x_2^2)^{\frac{p}{2}}
            \leq (p^{\frac{2}{p}} |x_2|^{\frac{2}{p}} + x_2^2)^{\frac{p}{2}}
            = (|x_2|^{\frac{2}{p}} (p^{\frac{2}{p}} + |x_2|^{2 - \frac{2}{p}}))^{\frac{p}{2}} \\
            &= |x_2| (p^{\frac{2}{p}} + |x_2|^{\frac{2(p-1)}{p}})^{\frac{p}{2}}
            = f(x) (p^{\frac{2}{p}} + |x_2|^{\frac{2(p-1)}{p}})^{\frac{p}{2}}.
        \end{align*}
        From these two inequalities, it is easy to follow that $x^*$ if a minimum of order $p$. Furthermore, similar to Example \ref{example:higher_order_convex}, Lemma \ref{lem:convex_higher_order} shows that \eqref{eq:higher_order_semismooth} holds for $f$. \\
        (i) For $p = 1$ and $\kappa \in (0,1)$, consider the sequences $(x^j)_j$, $(\eps_j)_j$ and $(\delta_j)_j$ given by 
        \begin{align*}
            x^j = (\kappa^j,0)^\top, \quad 
            \eps_j = \frac{\kappa^j}{\sqrt{2}}, \quad 
            \delta_j = 0, 
            \quad \forall j \in \N.
        \end{align*}
        See Figure \ref{fig:example_xj_speed_estim_equal_and_f_rate_sharp}(b) for a visualization. For $y = x^j + \eps_j (-\frac{1}{\sqrt{2}},\frac{1}{\sqrt{2}})^\top$ it holds $y \in \Bcl_{\eps_j}(x^j)$ and
        \begin{align*}
            |y_1| = \kappa^j - \eps_j \frac{1}{\sqrt{2}} = \frac{\kappa^j}{2} = |y_2|. 
        \end{align*}
        By construction of $f$, this means that $(0,1)^\top \in \partial f(y) \subseteq \partial_{\eps_j} f(x^j)$. Due to symmetry, we also obtain $(0,-1)^\top \in \partial_{\eps_j} f(x^j)$, such that $\min(\| \partial_{\eps_j} f(x^j) \|) = 0 = \delta_j$ for all $j \in \N$. In light of Theorem \ref{thm:xj_speed_higher_order}, we have
        \begin{align*}
            \| x^j - x^* \| = \| (\kappa^j, 0)^\top \| = \kappa^j = \frac{1}{\sqrt{2}} \eps_j,
        \end{align*}
        so \eqref{eq:xj_speed_higher_order_1} holds with $M = 1/\sqrt{2}$. \\
        (ii) For $p \geq 2$ and $\kappa \in (0,1)$, consider the sequences $(x^j)_j$, $(\eps_j)_j$ and $(\delta_j)_j$ given by 
        \begin{align*}
            x^j = (\kappa^j,0)^\top, \quad 
            \eps_j = 
            \begin{cases}
                (\kappa^p)^j, & j \text{ even} \\
                0, & j \text{ odd}
            \end{cases}, \quad 
            \delta_j = 
            \begin{cases}
                0, & j \text{ even} \\
                (\kappa^{p-1})^j, & j \text{ odd}
            \end{cases}, 
            \quad \forall j \in \N.
        \end{align*}
        If $j$ is even, then for the point $y = x^j + \eps_j (0,1)^\top$, it holds $y \in \Bcl_{\eps_j}(x^j)$ and 
        \begin{align*}
            p^{-1} |y_1|^p = p^{-1}  (\kappa^j)^p < (\kappa^p)^j = |y_2|.
        \end{align*}
        By construction of $f$, this means that $\nabla f(y) = (0,1)^\top \in \partial_{\eps_j} f(x^j)$. Due to symmetry, we also obtain $(0,-1)^\top \in \partial_{\eps_j} f(x^j)$, such that $\min(\| \partial_{\eps_j} f(x^j) \|) = 0 = \delta_j$. If $j$ is odd then $\eps_j = 0$, so
        \begin{align*}
            \| \partial_{\eps_j} f(x^j) \| = \| \{  \nabla f(x^j) \} \| = (\kappa^j)^{p-1} = \delta_j.
        \end{align*}
        In light of Theorem \ref{thm:xj_speed_higher_order}, we have $\| x^j - x^* \| = \| (\kappa^j, 0)^\top \| = \kappa^j$ and
        \begin{align*}
            \max(\eps_j^{\frac{1}{p}}, \delta_j^{\frac{1}{p-1}}) = \kappa^j = \| x^j - x^* \| \quad \forall j \in \N.
        \end{align*}
        Thus, for $M = 1$, we have equality in \eqref{eq:xj_speed_higher_order_2} (and the expression for which the maximum is attained alternates with $j$).
    \end{example}

    \begin{example} \label{example:f_rate_sharp}
        Consider the absolute value function
        \begin{align*}
            f : \R \rightarrow \R, \quad x \mapsto |x|
        \end{align*}
        with unique global minimum $x^* = 0$. Let $K \in \N$. Consider the sequences $(\eps_j)_j$ and $(\delta_j)_j$ given by $\delta_j = 0$ and
        \begin{align*}
            \eps_j = \frac{2}{(2K + 1)^j} \quad \forall j \in \N.
        \end{align*}
        It is possible to show that applying Algorithm \ref{algo:abstract_descent_method} for $x^0 = 1$ (with $W = \partial_{\eps_j} f(x^{j,i})$ in Step 2 and $t = \eps_j / \| v \|$ in Step 6) yields $N_j = K$,
        \begin{align*}
             x^{j,i} = \frac{1}{(2 K + 1)^{j-1}} - i \frac{2}{(2K + 1)^j}
             \ \text{ and } \ x^j = \frac{1}{(2 K + 1)^{j}}
             \quad \forall j \in \N, i \in \{1,\dots,K\}.
        \end{align*}
        In particular, Lemma \ref{lem:fz_rate} can be applied with $r(j) = 1/(2 K + 1)^{j}$ and $\Bar{N} = K$. The result is shown in Figure \ref{fig:example_xj_speed_estim_equal_and_f_rate_sharp}(c) (for $K = 10$). We see that up to a constant factor, the estimate \eqref{eq:fz_rate} is tight.
    \end{example}
    
\end{appendices}

\bibliography{references}

\end{document}